\documentclass{amsart}
\usepackage[T1]{fontenc}
\usepackage{enumerate}

\usepackage[dvips]{graphicx}
\usepackage{epsfig, latexsym, amsfonts, amsthm, amsmath, pstricks, mathtools}
\usepackage[mathscr]{euscript}

\renewenvironment{proof}[1][]{\vskip-\lastskip\par\vskip6pt plus2pt
minus0pt\par%
\noindent\textit{Proof.}\enspace\ignorespaces}{\hfill$\Box$\par\vskip6pt
plus2pt minus0pt}

\numberwithin{equation}{section}

\newtheorem{theorem}[equation]{Theorem}
\newtheorem{lemma}[equation]{Lemma}
\newtheorem{proposition}[equation]{Proposition}
\newtheorem{corollary}[equation]{Corollary}
\theoremstyle{definition}
\newtheorem{definition}[equation]{Definition}

\newtheorem{remark}[equation]{Remark}

\newtheorem{observations}[equation]{Observations}

\newtheorem{notation}[equation]{Notation}
\newtheorem{conventions}[equation]{Conventions}

\newtheorem{example}[equation]{Example}

\begin{document}
\title{Veldkamp Quadrangles and Polar Spaces}

\author{Bernhard M\"uhlherr and Richard M. Weiss}
\address{Mathematisches Institut \\
         Universit\"at Giessen,
         35392 Giessen, Germany}
\email{bernhard.m.muehlherr@math.uni-giessen.de}
\address{Department of Mathematics \\
         Tufts University \\
         Medford, MA 02155, USA}
\email{rweiss@tufts.edu}


\date{\today}

\begin{abstract}
Veldkamp polygons are certain graphs $\Gamma=(V,E)$ such that for each $v\in V$,
$\Gamma_v$ is  endowed with a symmetric anti-reflexive
relation $\equiv_v$. These relations are all trivial if and only if 
$\Gamma$ is a thick generalized polygon. A Veldkamp polygon is called flat if 
no two vertices have the same set of vertices that are opposite in a natural sense.
We explore the connection between Veldkamp 
quadrangles and polar spaces. Using this connection, we give the complete classification
of flat Veldkamp quadrangles in which some but not all of the relations $\equiv_v$ are trivial.
\end{abstract}

\maketitle


\section{Introduction}\label{intro}
In \ref{toy6}, we introduce the notion of a Veldkamp polygon. A Veldkamp polygon
is a graph $\Gamma=(V,E)$ such that for each vertex $v\in V$, the set $\Gamma_v$ of vertices
adjacent to $v$ is endowed with a symmetric, anti-reflexive relation $\equiv_v$ satisfying
several axioms. The relations $\equiv_v$ are called the {\it local opposition relations}
of $\Gamma$. If $(u,v,w)$ is a path of length~$2$ and $u\equiv_vw$, we say that
the path $(u,v,w)$ is {\it straight}.
An arbitrary path is straight if every subpath of length~$2$ is straight. Thus every path is
straight if and only if the local opposition relations $\equiv_v$ are all trivial
(in the sense that $u\equiv_vw$ for all distinct $u,w\in\Gamma_v$).
This is the case if and only if $\Gamma$ is a thick generalized polygon. 

The thick generalized polygons that arise in the theory of algebraic groups all satisfy
the Moufang condition. This is a transitivity condition on the root groups. There
are natural extensions of the notion 
of a root group and the Moufang condition to Veldkamp polygons. A Veldkamp polgon
satisfying this Moufang condition is the same thing as a Tits polygon, a notion 
introduced in \cite{TP}. 

There is a standard way to construct a Tits polygon from an irreducible spherical building of 
rank at least~$3$ and a 
Tits index that is compatible in a suitable sense.
Every irreducible spherical building of rank at least~$3$ has at least one
compatible Tits index. Thus the theory of Tits polygons allows us, in particular,
to study an arbitrary spherical building as a rank~$2$ structure on which the methods
in \cite{TW} can be brought to bear (as in \cite{TP}--\cite{orth}). 
This was our original reason for 
introducing Tits polygons. In \cite{root}, we also showed that there is an equivalence between
Tits polygons and a certain class of root graded groups. 

In \cite[Def.~3.1]{V2}, Veldkamp introduced the notion of a Veldkamp plane. (Veldkamp 
called them Barbilian planes.)
This notion arose in his work on projective geometry over rings.
In the same paper Veldkamp showed the importance of the notion of
the stable rank of a ring (a notion he found in $K$-theory) in ring-geometry,
a subject first studied by Corrado Segre early in the 20th century (see \cite{V1}).
Veldkamp was able to show that many of the basic results of projective
geometry, with suitable modifications, hold over rings of stable rank~$2$. In particular, he showed that
the geometries that arise from a free module of rank~$3$ over an associative ring of stable
rank~$2$ and Veldkamp planes satisfying something slightly stronger than the Moufang condition 
(see \cite[Def.~4.16]{V2}) are essentially the same. This result was subsequently 
strengthened first by Faulkner in \cite[Thm~9]{faulkner} and then by 
ourselves in \cite[Thms.~7.2 and 7.3]{triangles}.

A Veldkamp plane as defined by Veldkamp is the same thing as a
Veldkamp $n$-gon as defined in \ref{toy6} with $n=3$, 
where the local opposition relations play the role of the relation
$\not\approx$ in \cite[Def.~3.1]{V2}. 
Thus our notion of a Veldkamp polygon is a natural generalization of the notion of a
Veldkamp plane.

We say that 
a Veldkamp polygon is {\it green} if some (or all) of its local opposition relations are trivial,
and we call a Veldkamp polygon {\it flat} if no two vertices have the same set of vertices 
that are opposite as defined in \ref{toy7}.
All generalized polygons are both green and flat.

In this paper, the focus is on Veldkamp quadrangles (i.e. a Veldkamp $n$-gon with $n=4$). 
We prove four theorems.
In our first theorem (see \ref{fla14}, \ref{fla20} and \ref{fla88}), 
we show that every green Veldkamp quadrangle $\Gamma$ has a canonical
flat quotient $\bar\Gamma$ (which is also a green Veldkamp quadrangle) and a natural 
surjective homomorphism
$\pi$ from $\Gamma$ to $\bar\Gamma$.

In our second theorem (see \ref{hau12} and \ref{hau13}), 
we show that there is a one-to-one correspondence (up to isomorphism) between
flat green Veldkamp quadrangles and non-degenerate polar spaces with thick lines.

The notion of a polar space
was also introduced by Veldkamp. Polar spaces of rank~$2$ are generalized quadrangles.
Non-degenerate polar spaces of finite rank at least~$3$ and characteristic
different from ~$2$ were classified by Veldkamp in \cite{V}. 
This result was extended to arbitrary characteristic
by Tits in \cite{spherical}. Several people contributed to the extension of Tits' result 
to include polar spaces of infinite rank. We cite \cite{shult} for this part of the 
classification. 

In our third theorem (see \ref{chin6}), 
we use the classification of non-degenerate polar spaces of rank at least~$3$
to show that every flat green Veldkamp quadrangle is either a generalized quadrangle or
a Tits quadrangle obtained from a quadratic space or a pseudo-quadratic space (which we 
interpret to include the symplectic case) or a Tits quadrangle of type~$D_3$ or type~$E_7$
as defined in \ref{chin50} and \ref{chin60}.
In particular, a flat green Veldkamp quadrangle that is not a generalized quadrangle always
satisfies the Moufang condition.

There are, of course, also many generalized quadrangles
that do not satisfy the Moufang condition.
Generalized quadrangles that do satisfy the Moufang condition were classified in \cite[17.4]{TW}.
There are six families. Three of them have a natural description in terms of polar spaces,
but three of them do not; see \ref{chin6a}. 

In our fourth theorem (see \ref{chin3} and \ref{chin5}), 
we show that to every non-degenerate quadratic space and
every non-degenerate pseudo-quadratic space (with infinite Witt index allowed), 
there does, in fact, exist a corresponding flat green Veldkamp quadrangle.

$$*\quad\qquad*\quad\qquad*$$
\begin{conventions}\label{pip40}
As in \cite{TW}, we set 
$$[a,b]=a^{-1}b^{-1}ab$$
for elements $a,b$ in a group and we apply permutations from left to right.
\end{conventions}

\medskip
\noindent
{\sc Acknowledgment:} The work of the first author was partially supported
by DFG Grant MU1281/7-1 and the work of the second author was partially supported
by Simons Foundation Collaboration Grant~516364.

\section{Veldkamp polygons}\label{poly}

In this section we introduce the notion of a Veldkamp polygon. In \ref{toy34},
we describe the connection between Veldkamp polygons and generalized polygons.

\begin{definition}\label{toy1}
An {\it opposition relation} on a set $X$ is a relation on $X$ 
that is symmetric and anti-reflexive. Let $\equiv$ be an opposition relation on $X$.
We say that two elements of $X$ are {\it opposite} if $x\equiv y$.
We say that $\equiv$ is {\it trivial} if every two distinct elements of $X$ are opposite.
\end{definition}

\begin{definition}\label{toy2}
Let $\equiv$ be an opposition relation on a set $X$. For each $k\ge1$,
we say that $\equiv$ is {\it $k$-plump} if for every subset $S$ of $X$
such that $|S|\le k$, there exists an element of $X$ that is opposite
every element of $S$. Note that if $\equiv$ is 
$k$-plump for some $k\ge1$, then it is also $(k-1)$-plump and $|X|\ge k+1$.
\end{definition}

\begin{notation}\label{toy3}
By {\it graph} we mean a pair $(V,E)$ consisting of a
set $V$ and a set $E$ of $2$-element subsets of $V$. The elements of
$V$ and $E$ are called {\it vertices} and {\it edges}. For each $v\in V$,
we let $\Gamma_v$ denote the set of {\it neighbors} of $v$. In other words,
$$\Gamma_v=\{u\in V\mid \{u,v\}\in E\}.$$
Let $\Gamma=(V,E)$ be a graph. A {\it subgraph} of $\Gamma$ is a pair $(V_1,E_1)$,
where $V_1$ is a non-empty subset of $V$ and $E_1$ is a subset of $\{e\in E\mid e\subset V_1\}$.
For each non-empty subset $V_1$ of $V$, the 
subgraph of $\Gamma$ {\it spanned by $V_1$} is the subgraph $(V_1,E_1)$, where 
$E_1=\{e\in E\mid e\subset V_1\}$.
An {\it $s$-path} in $\Gamma$ is a sequence $(x_0,x_1,\dots,x_s)$ 
of $s+1$ vertices $x_i$ for
some $s\ge0$ such that $x_{i-1}\in\Gamma_{x_i}$ for all $i\in[1,s]$.
and $x_{i-2}\ne x_i$ for all $i\in[2,s]$. (Here $[i,j]$ denotes the closed interval $i,i+1,\ldots,j$
if $i\le j$ and $[i,j]=\emptyset$ if $i>j$.) A {\it path} is an $s$-path for some
$s\ge0$; the parameter $s$ is called the {\it length} of the path. 
The graph $\Gamma$ is {\it connected} if there is a path from any given vertex
to any other given vertex. The distance ${\rm dist}(x,y)$ between two 
vertices is the minimal length of a path from $x$ to $y$ (assuming that there is one).
For each $v\in V$, let 
$$\Gamma_m(v)=\{u\in V\mid {\rm dist}(u,v)=m\}.$$
Thus $\Gamma_1(v)=\Gamma_v$ for all $v\in V$.
\end{notation}

\begin{notation}\label{toy3x}
An $s$-path is {\it closed} if its first and last vertices are the same and $s\ge3$. 
An {\it $s$-circuit} is the subgraph with vertex set $\{x_1,\ldots,x_s\}$ 
and edge set $\big\{\{x_{i-1},x_i\}\mid i\in[1,s]\big\}$ for some closed
$s$-path $(x_0,x_1,\ldots,x_s)$. A {\it circuit} is an $s$-circuit for some $s\ge3$. 
\end{notation}

\begin{notation}\label{toy4}
Let $\Gamma=(V,E)$ be a graph endowed with an opposition relation $\equiv_v$
on $\Gamma_v$ for each $v\in V$.
We say that two vertices 
$x,y$ are {\it opposite at $v$} if $x,y\in\Gamma_v$ and $x\equiv_v y$. 
A path $(v_0,v_1,\ldots,v_s)$ is called {\it straight} if 
$v_{i-1}$ is opposite $v_{i+1}$ at $v_i$ for all $i\in[1,s-1]$.
A circuit is {\it straight} if every path it contains is straight.
We say that $\Gamma$ is {\it $k$-plump} for some $k\ge1$ if all its local 
opposition relations are $k$-plump as defined in \ref{toy2}. 
\end{notation}

\begin{definition}\label{toy5}
A {\it Veldkamp graph} is a graph $\Gamma=(V,E)$ 
endowed with a $2$-plump opposition relation on $\Gamma_v$ for each $v\in V$.
\end{definition}

\begin{remark}\label{toy5x}
Since every Veldkamp graph is $2$-plump (and therefore also $1$-plump),
it follows that for every $s\ge0$,
a straight $s$-path in a Veldkamp graph can be extended to a straight $(s+1)$-path.
\end{remark}

\begin{definition}\label{toy6}
Let $n\ge2$. A {\it Veldkamp $n$-gon} is a Veldkamp graph $\Gamma$ 
such that the following hold:
\begin{enumerate}[\rm(VP1)]
\item $\Gamma$ is connected and bipartite.
\item For each $k\in[1,n-1]$ and each straight $k$-path $\alpha=(v_0,\ldots,v_k)$,
$\alpha$ is the unique straight path from $v_0$ to $v_k$ of length at most~$k$.
\item Every straight $(n+1)$-path is contained in a straight $2n$-circuit.
\end{enumerate}
Note that by (VP2), the straight $2n$-circuit in (VP3) is unique. A {\it Veldkamp polygon}
is a Veldkamp $n$-gon for some $n\ge2$.
\end{definition}

\begin{definition}\label{toy7}
Let $\Gamma=(V,E)$ be a Veldkamp $n$-gon for some $n\ge2$. A {\it root} in $\Gamma$ is  
a straight $n$-path. Two vertices $x$ and $y$
are {\it opposite} if there is a root from $x$ to $y$
and two edges $e,f$ of $\Gamma$ are {\it opposite} if there exists a straight $(n+1)$-path
whose first and last edges are $e$ and $f$. By (VP2), neither a vertex nor an edge
can be opposite itself, so both of these relations are, in fact, opposition relations
(on $V$ and on $E$). The set of vertices opposite
a given vertex $x$ is denoted by $x^{\rm op}$ and $\Gamma$ is called {\it flat}
if $x^{\rm op}=y^{\rm op}$ for $x,y\in V$ implies that $x=y$.
\end{definition}

Suppose for the rest of this section that $\Gamma=(V,E)$ is a Veldkamp $n$-gon for some $n\ge2$.

\begin{proposition}\label{toy8}
Suppose that $x,y\in V$ are opposite and that $z\in\Gamma_y$. Then there exists a unique
root $(x,\ldots,z,y)$ from $x$ to $y$ containing $z$. 
\end{proposition}

\begin{proof}
Uniqueness holds by (VP2). Since $x$ and $y$ are opposite, there is a root
$(x,\ldots,w,y)$ from $x$ to $y$. Since $\Gamma$ is $2$-plump, 
we can choose $u\in\Gamma_y$ opposite $w$ and $z$
at $y$. Thus $(x,\ldots,w,y,u)$ is a straight $(n+1)$-path. By (VP3),
it follows that there exists a root $(x,\ldots,u,y)$ from $x$ to $y$ that contains
$u$. Then $(x,\ldots,u,y,z)$ is a straight $(n+1)$-path and the claim follows by a second
application of (VP3).
\end{proof}

\begin{remark}\label{toy9xx}
Suppose that $u,v\in V$ and ${\rm dist}(u,v)$ is even. Since $\Gamma$ is bipartite,
every vertex in $u^{\rm op}$ is at even distance to every vertex in $v^{\rm op}$
\end{remark}

\begin{proposition}\label{toy9x}
For each pair $e_1,e_2$ of edges of $\Gamma$, there exists an
edge that is opposite both $e_1$ and $e_2$.
\end{proposition}

\begin{proof}
Let $(x_0,\ldots,x_k)$ be a $k$-path in $\Gamma$ for some $k\ge1$ and 
let $c_i$ denote the edge $\{x_{i-1},x_i\}$ for each $i\in[1,k]$.
Since $\Gamma$ is connected, it will suffice 
to show that there exists an edge opposite both $c_1$ and $c_k$.
We proceed by induction with respect to $k$. By \ref{toy5x},
the claim holds for $k=1$. Suppose that $k\ge2$ and that
$f$ is an edge that is opposite $c_1$ and $c_{k-1}$.
Let $y$ be the unique vertex of $f$ that is opposite $x_{k-1}$.
Since $\Gamma$ is bipartite, 
$y$ is opposite to a unique vertex on $c_1$. By \ref{toy9xx},
this vertex is at even distance from $x_{k-1}$. Thus $y$ is opposite $x_p$,
where $p=0$ if $k$ is odd and $p=1$ if $k$ is even.

Let $\alpha=(x_{k-1},x_k,\ldots,y)$ be the unique root
from $x_{k-1}$ to $y$ containing $x_k$ and let $u$ be unique the element of $\Gamma_y$ contained
in $\alpha$. If $k$ is odd, let
$\beta=(x_0,x_1,\ldots,y)$ be the unique root from $x_0$ to $y$ containing $x_1$
and if $k$ is even, let $\beta=(x_1,x_0,\ldots,y)$ be the unique root from 
$x_1$ to $y$ containing $x_0$.
In both cases, let $v$ be unique element of $\Gamma_y$ contained in $\beta$.
Since $\Gamma$ is $2$-plump, we can choose a vertex $z\in\Gamma_y$ that 
is opposite $u$ and $v$. Let $e=\{y,z\}$. Attaching $z$ at the end of $\alpha$ and at the
end of $\beta$ yields straight $(n+1)$-paths. Hence $e$ is an edge opposite both $c_1$ and $c_k$.
\end{proof}

\begin{corollary}\label{toy9}
Let $u,v$ be vertices such that ${\rm dist}(u,v)$ is even.
Then $u^{\rm op}\cap v^{\rm op}\ne\emptyset$.
\end{corollary}

\begin{proof}
Let $e$ be an edge containing $u$ and let $f$ be an edge containing $v$.
By \ref{toy9x}, there exists an edge $g$ opposite both $e$ and $f$.
Let $x$ be the unique vertex of $g$ that is opposite $u$ 
and let $y$ be the unique vertex of $g$ that is opposite $v$.
By \ref{toy9xx}, $x=y$ and thus $x\in u^{\rm op}\cap v^{\rm op}$.
\end{proof}

\begin{definition}\label{toy39}
We recall that a {\it generalized $n$-gon} for some $n\ge2$ is a bipartite graph $\Omega$ such that
the following hold:
\begin{enumerate}[\rm(i)]
\item ${\rm diam}(\Omega)=n$.
\item $\Omega$ contains no circuits of length less than $2n$.
\item $|\Omega_v|\ge2$ for all vertices $v$.
\end{enumerate}
A generalized $n$-gon $\Omega$ is {\it thick} if $|\Omega_v|\ge3$ for all vertices $v$.

\end{definition}

\begin{proposition}\label{toy34}
If the local opposition relations are all trivial, then $\Gamma$ is a thick generalized $n$-gon,
and every thick generalized $n$-gon, when endowed with the trivial local opposition relation
at every vertex, is a Veldkamp $n$-gon.
\end{proposition}

\begin{proof}
Suppose that all local opposition relations of $\Gamma$ are trivial. Then every path is straight.
By (VP1), $\Gamma$ is connected. By (VP3), therefore, ${\rm diam(\Gamma)}\le n$.
By (VP2), $\Gamma$ contains no circuits
of length less than~$2n$. This implies that ${\rm diam}(\Gamma)\ge n$.
Since $\Gamma$ is $2$-plump, we have
$|\Gamma_x|\ge3$ for all vertices $x$. It follows that $\Gamma$ is a thick generalized $n$-gon.

Suppose, conversely, that $\Omega$ is a thick generalized
$n$-gon endowed with the trivial local opposition relation at every vertex, 
so that every path is straight. 
Since $\Omega$ is thick, it is $2$-plump.
By \ref{toy39}(ii), $\Omega$ satisfies (VP2) and by \ref{toy39}(i),
it follows that (VP1) and (VP3) hold. Thus $\Omega$ is a Veldkamp $n$-gon.
\end{proof}

\begin{proposition}\label{toy35}
Suppose that $(x_0,x_1,\ldots,x_n)$ and $(y_0,y_1,\ldots,y_n)$ are two roots
with $x_0=y_0$ and $x_n=y_n$. Then $x_1$ and $y_1$ are opposite at $x_0$
if and only if $x_{n-1}$ and $y_{n-1}$ are opposite at $x_n$.
\end{proposition}

\begin{proof}
Suppose that $x_{n-1}$ and $y_{n-1}$ are opposite at $x_n$. Then
$$(x_0,x_1,\ldots,x_{n-1},x_n,y_{n-1})$$ 
is a straight $(n+1)$-path.
By (VP3), there exists a root $\alpha:=(z_0,z_1,\ldots,z_n)$ such that 
$z_0=x_0$, $z_{n-1}=y_{n-1}$, $z_n=x_n$ and $z_1$ is opposite $x_1$ at $x_0$.
By (VP2), it follows that $\alpha=(y_0,y_1,\ldots,y_n)$. Hence $x_1$ is opposite $y_1$ at $x_0$.
The converse holds by symmetry.
\end{proof}

\begin{proposition}\label{toy36}
Suppose that $\equiv_v$ is trivial for some $v\in V$. 
Then $\equiv_w$ is trivial for all $w$
at even distance from $v$ and if, in addition, $n$ is odd, then $\equiv_w$ is trivial
for all $w\in V$.
\end{proposition}

\begin{proof}
Let $(x_0,\ldots,x_n)$ be a root.
By \ref{toy35}, $\equiv_{x_0}$ is trivial if and only if $\equiv_{x_n}$ is trivial.
Thus if ${\rm dist}(v,w)$ is even for some $w\in V$, then by \ref{toy9},
$\equiv_w$ is trivial. 
If $n$ is odd, then by (VP1), every vertex of $V$ is at even distance from 
$x_0$ or from $x_n$, so if one local opposition relation is trivial, then they all are.
\end{proof}

\section{Veldkamp quadrangles}

In this section we focus on the case $n=4$. Let $\Gamma=(V,E)$ be a Veldkamp quadrangle.
Since $\Gamma$
is connected and bipartite, there is a unique decomposition of $V$ into two subsets
such that every edge contains one vertex from each of them. We choose one of these
two subsets and call it $P$ and the other we call $L$.
We refer to the elements of $P$ as {\it points} and to the elements of $L$ as {\it lines}.

\begin{notation}\label{toy10}
We call $\Gamma$ {\it green} if the local opposition relations at lines are all trivial.
\end{notation}

We assume for the rest of this section
that our Veldkamp quadrangle $\Gamma$ is green.

\begin{proposition}\label{fla3}
$\Gamma$ does not contain circuits of length~$4$.
\end{proposition}

\begin{proof}
Let $(a,x,b,y,c)$ be a $4$-path with $c=a\in P$. Since $\Gamma$ is green,
$(a,x,b)$ and $(a,y,b)$ are distinct straight paths from $a$ to $b$. 
By (VP2), this is not allowed.
\end{proof}

\begin{notation}\label{fla3x}
Let $x,y$ be vertices such that ${\rm dist}(x,y)=2$. By \ref{fla3},
there is a unique vertex in $\Gamma_x\cap\Gamma_y$. We denote this vertex by $x\wedge y$.
\end{notation}

\begin{proposition}\label{fla40}
A $6$-circuit does not contain any straight $2$-paths
$(x,a,y)$ such that $a\in P$.
\end{proposition}

\begin{proof}
Let $(a,x,b,y,c,z,a)$ be a closed $6$-path with $a\in P$. It suffices to show that
$(x,b,y)$ is not straight. Suppose that it is. 
Then $(a,x,b,y,c)$ is a root, so by \ref{toy7}, $a$ and $c$ are opposite. By \ref{toy8},
there is a root $(a,u,d,z,c)$ from $a$ to $c$ that contains $z$. 
Thus $(a,z)$ and $(a,u,d,z)$ are both straight paths from $a$ to $z$. By (VP2),
this is not allowed. 
\end{proof}

\begin{proposition}\label{toy20}
If $a,b\in P$ are opposite, then ${\rm dist}(a,b)=4$.
\end{proposition}

\begin{proof}
Suppose that $a,b\in P$ are opposite. 
By \ref{toy7}, there is a root from $a$ to $b$. In particular,
${\rm dist}(a,b)\le4$. By \ref{fla3}, $a\ne b$ and then by \ref{fla40}, ${\rm dist}(a,b)\ne2$.
Hence ${\rm dist}(a,b)=4$.
\end{proof}

\begin{proposition}\label{toy21}
${\rm diam}(\Gamma)=4$.
\end{proposition}

\begin{proof}
By \ref{toy20}, it suffices to show that ${\rm dist}(a,u)\le3$ for all $a\in P$ and $u\in L$.
Let $a\in P$ and $u\in L$.  Choose $b\in\Gamma_u$. By \ref{toy9}, we can
choose $c\in a^{\rm op}\cap b^{\rm op}$. By \ref{toy8}, there 
is a root $(b,x_1,x_2,x_3,c)$ from $b$ to $c$ such that $x_1=u$ and 
a root $(a,y_1,y_2,y_3,c)$ from $a$ to $c$ such that
$y_3=x_3$. If $y_2=x_2$, then $(a,y_1,x_2,x_1)$ is a $3$-path from $a$ to $u$.
Suppose that $y_2\ne x_2$. Then $(a,y_1,y_2,x_3,x_2,x_1)$ is a 
$5$-path from $a$ to $u$. Since $\Gamma$ is green, this path is straight.
By (VP3), therefore, ${\rm dist}(a,u)\le3$ also in this case.
\end{proof}

\begin{proposition}\label{toy22}
Let $a,b\in P$ and suppose that ${\rm dist}(a,b)=4$ but $a$ and $b$ are not opposite.
Then $a^{\rm op}=b^{\rm op}$ and $\Gamma_2(a)=\Gamma_2(b)$.
\end{proposition}

\begin{proof}
Suppose there exists $c\in \Gamma_2(b)\cap a^{\rm op}$ and let $x=b\wedge c$.
By \ref{toy8}, there exists a root $(a,y_1,y_2,y_3,c)$ from $a$ to $c$ 
with $y_3=x$. Since $\Gamma$ is green, it follows that  $(a,y_1,y_2,y_3,b)$ is a root.
Since $a\not\in b^{\rm op}$ by hypothesis, we obtain a contradition. Therefore
\begin{equation}\label{toy22a}
\Gamma_2(b)\cap a^{\rm op}=\emptyset.
\end{equation}

Now choose $x\in\Gamma_b$. We claim that 
\begin{equation}\label{toy22b}
\Gamma_x\backslash\{b\}\subset \Gamma_2(a).
\end{equation}
By \ref{toy9}, we can choose $c\in a^{\rm op}\cap b^{\rm op}$.
Let $(b,x_1,x_2,x_3,c)$ be the unique 
root from $b$ to $c$ such that $x_1=x$ and let $(c,y_1,y_2,y_3,a)$ be the 
unique root from $c$ to $a$ such that $y_1=x_3$. 
By \eqref{toy22a}, $(a,y_3,y_2,y_1,x_2)$ is not a root. Since $\Gamma$ is green,
it follows that $x_2=y_2$.

Choose $z\in\Gamma_c$ opposite $x_3$ at $c$. Let $(c,z_1,z_2,z_3,b)$ and 
$(c,z_1',z_2',z_3',a)$ be the unique roots from $c$ to $b$ and from $c$ to $a$
such that $z_1=z_1'=z$. The closed $8$-path $(b,z_3,z_2,z,c,x_3,x_2,x,b)$ 
is straight. By \ref{toy35}, $(x,b,z_3)$ is also straight. 

If $z_2\ne z_2'$, then $(a,z_3',z_2',z_1',z_2)$
is a root and hence $z_2\in \Gamma_2(b)\cap a^{\rm op}$. By \eqref{toy22a}, it 
follows that $z_2=z_2'$. 

Let $d\in\Gamma_x\backslash\{x_2,b\}$. To prove \eqref{toy22b}, we need to 
show that $d\in \Gamma_2(a)$. The path $(d,x,x_2,x_3,c)$ is a root.
Let $(c,w_1,w_2,w_3,d)$ be the unique root from $c$ to $d$ such that $w_1=z$.
By \ref{fla40}, we have $w_2\ne z_2$ since otherwise the straight $2$-path $(x,b,z_3)$ is 
contained in the closed $6$-path $(x,b,z_3,z_2,w_3,d,x)$. Therefore
$(a,z_3',z_2,z,w_2)$ is a root. Hence $w_2$ and $a$ are opposite.
Let $(w_2,v_1,v_2,v_3,a)$ be the unique root from $w_2$ to $a$ such that $v_1=w_3$.
If $v_2\ne d$, then $(d,w_3,v_2,v_3,a)$ is a root, but this implies
that $d\in \Gamma_2(b)\cap a^{\rm op}$. By \eqref{toy22a}, we must have
$v_2=d$. Therefore $d\in \Gamma_2(a)$. 
Thus \eqref{toy22b} holds as claimed.

By \ref{toy21}, \eqref{toy22a} and 
symmetry, we have $\Gamma_2(a)=\Gamma_2(b)$. Thus
\begin{equation}\label{toy22c}
\Gamma_2(e)=\Gamma_2(f)
\end{equation} 
for all $e,f\in P$ such that ${\rm dist}(e,f)=4$ and $e\not\in f^{\rm op}$.
It remains to show that
$a^{\rm op}=b^{\rm op}$. Suppose that $c\in a^{\rm op}\backslash b^{\rm op}$ and 
let $(a,x_1,x_2,x_3,c)$ be a root from $a$ to $c$. Thus
$z\in \Gamma_2(c)\cap a^{\rm op}$ for all $z\in\Gamma_{x_3}\backslash\{x_2,c\}$.
Therefore 
\begin{equation}\label{toy33d}
\Gamma_2(c)\cap a^{\rm op}\ne\emptyset.
\end{equation}
Since $a$ and $b$ are not opposite, we have $c\ne b$. By \eqref{toy22a},
${\rm dist}(b,c)\ne2$. Thus ${\rm dist}(b,c)=4$ by \ref{toy21}.
Hence $\Gamma_2(a)=\Gamma_2(b)=\Gamma_2(c)$ by \eqref{toy22c}.
It follows that $\Gamma_2(a)\cap a^{\rm op}\ne\emptyset$ by \eqref{toy33d}. By \ref{toy20},
this is impossible. With this contradiction, we conclude
that $a^{\rm op}=b^{\rm op}$.
\end{proof}

\begin{corollary}\label{toy25}
Suppose that $\Gamma$ is flat as 
defined in {\rm\ref{toy7}} and let $a,b\in P$. Then $a$ and $b$ are opposite
if and only if ${\rm dist}(a,b)=4$.
\end{corollary}

\begin{proof}
This holds by \ref{toy7}, \ref{toy20} and \ref{toy22}.
\end{proof}

\section{Flat quotients}\label{quo}

Our goal in this section is to show that every green Veldkamp quadrangle has a 
canonical flat quotient. The precise result is formulated in \ref{fla14} and \ref{fla88}.

For the rest of this section, we assume that $\Gamma$ is a green Veldkamp quadrangle
and continue with the notation in the previous section.

\begin{notation}\label{fla1}
A {\it weed} is a non-straight $4$-path $(a,x,b,y,c)$ such that 
$a,b,c\in P$ and ${\rm dist}(a,c)=4$.
\end{notation}

\begin{remark}\label{fla50}
Note that a root cannot also be a weed since roots are straight
and weeds are not.
\end{remark}

\begin{remark}\label{fla1x}
Let $\alpha=(a,x,b,y,c)$ be a weed. The point $b$ is called the {\it middle point} of 
$\alpha$. By \ref{fla3}, the middle point of a weed $\alpha$ is uniquely determined by the two
lines contained in $\alpha$. 
\end{remark}

\begin{notation}\label{fla5}
Let $a\simeq b$ for $a,b\in P$ if either $a=b$ or there is a weed that begins at $a$ and ends at $b$.
\end{notation}

\begin{proposition}\label{fla10}
Suppose that $u\in v^{\rm op}$ for some $u,v\in P\cup L$ and $a\simeq b$
for some $a,b\in P$ such that $a\ne b$. Then the following hold:
\begin{enumerate}[\rm(i)]
\item Every $4$-path from $u$ to $v$ is a root and for every $w\in\Gamma_u$,
there is a unique $4$-path from $u$ to $v$ that contains $w$.
\item For every $x\in\Gamma_a$, there exists a $4$-path from $a$ to $b$ containing $x$,
every $4$-path from $a$ to $b$ is a weed and $a\not\in b^{\rm op}$.
\end{enumerate}
\end{proposition}

\begin{proof}
Let $w\in\Gamma_u$. By \ref{toy20} and \ref{toy21}, we have ${\rm dist}(w,v)=3$. Hence
there exists a $4$-path
$\alpha=(u,p,q,w,v)$ from $u$ to $v$ that contains $w$.
By \ref{toy8}, there exists a root $\beta=(u,m,n,w,v)$ from
$u$ to $v$ that passes through $w$. By \ref{fla40}, $q=n$. By \ref{fla3}, therefore,
$\alpha=\beta$. Hence $\alpha$ is a root and $\alpha$ is the only $4$-path from 
$u$ to $v$ that contains $w$. Thus (i) holds.

Let $x\in\Gamma_a$.
Since $a\ne b$, there exists a weed $\alpha$ from $a$ to $b$.
Hence ${\rm dist}(a,b)=4$ and by \ref{toy21}, there exists
a $4$-path $\beta$ from $a$ to $b$ containing $x$. Applying (i), we observe
that since $\alpha$ is not
a root (by \ref{fla50}), neither is $\beta$ a root. Hence $\beta$ is a weed.
Thus (ii) holds.
\end{proof}

\begin{proposition}\label{fla5x}
Let $a,b\in P$. Then the following hold:
\begin{enumerate}[\rm(i)]
\item $a\simeq b$ if and only if $a^{\rm op}=b^{\rm op}$.
\item If $a\simeq b$, then $\Gamma_2(a)=\Gamma_2(b)$.
\end{enumerate}
\end{proposition}

\begin{proof}
It suffices to assume that $a\ne b$.
Suppose that $a\simeq b$. Then ${\rm dist}(a,b)=4$ but 
by \ref{fla10}(ii), $a$ and $b$ are not opposite.
By \ref{toy22}, therefore, $a^{\rm op}=b^{\rm op}$ and $\Gamma_2(a)=\Gamma_2(b)$.
Suppose, conversely, that $a^{\rm op}=b^{\rm op}$.
If ${\rm dist}(a,b)=2$, then (because $\Gamma$ is green) there exists a root $\alpha$ starting at $a$ 
and containing $b$, but then the last vertex of $\alpha$ is contained in 
$a^{\rm op}\backslash b^{\rm op}$. Hence ${\rm dist}(a,b)=4$ by \ref{toy21}. Let $\beta$
be a $4$-path from $a$ to $b$. Since
$a$ is not opposite itself, it is not opposite $b$. It follows that $\beta$ is a weed
and hence $a\simeq b$. 
\end{proof}

\begin{notation}\label{fla70}
By \ref{fla5x}(i), $\simeq$ is an equivalence relation. We denote the equivalence class
of a point $a$ by $[a]$.
\end{notation}

\begin{notation}\label{fla6}
Let $x\sim y$ for $x,y\in L$ if there is a weed containing both $x$ and $y$.
Thus if $x\sim y$, there exist $a,b,c\in P$ such that $(a,x,b,y,c)$ is a weed.
Let $\approx$ denote the transitive-reflexive closure of the relation $\sim$ on $L$.
For each $x\in L$, let $[x]$ denote the equivalence class containing $x$.
(Thus all the equivalence classes are singletons if there are no weeds.)
\end{notation}

\begin{proposition}\label{fla11}
Let $x,y,z\in\Gamma_a$ for some $a\in P$. Suppose that $x$ is opposite $y$ at $a$
and that $y\sim z$. Then $x$ is opposite $z$ at $a$.
\end{proposition}

\begin{proof}
Since $y\sim z$, there exist $b,c\in P$ such that $(c,z,a,y,b)$ 
is a weed. By \ref{fla5} and \ref{fla5x}(ii), we have 
$b\simeq c$ and $\Gamma_2(b)=\Gamma_2(c)$.
Choose $d\in\Gamma_x\backslash\{a\}$. By \ref{fla3}, the points $b,c,d$ are pairwise
distinct. Since $x$
and $y$ are opposite at $a$, $(d,x,a,y,b)$ is a root and thus $b$ and $d$ are opposite.
By \ref{toy20}, therefore, $d\not\in \Gamma_2(b)$.
Hence $d\not\in\Gamma_2(c)$. By \ref{toy21}, we have ${\rm dist}(c,d)=4$.
Suppose now that $x$ and $z$ are not opposite at $a$. Then 
$(c,z,a,x,d)$ is a weed and hence $c\simeq d$. Since $b\simeq c$,
it follows that $b\simeq d$.
By \ref{fla10}(ii), this is impossible. With this contradiction, we conclude
that $x$ is opposite $z$ at $a$.
\end{proof}

\begin{proposition}\label{fla11x}
Let $x,y\in L$ and suppose that $x\approx y$. Then $x^{\rm op}=y^{\rm op}$.
\end{proposition}

\begin{proof}
It suffices to assume that $x\sim y$ and, by symmetry, to show that $x^{\rm op}\subset y^{\rm op}$.
Let $a=x\wedge y$ and choose
$z\in x^{\rm op}$. By \ref{fla10}(i), there exists a root $(x,a,u,c,z)$
from $x$ to $z$ that contains $a$. The lines $x$ and $u$ are opposite at $a$.
By \ref{fla11}, it follows that $u$ is opposite $y$ at $a$. Thus $(y,a,u,c,z)$
is also a root. Hence $z\in y^{\rm op}$. 
\end{proof}

We will need the following result in the proof of \ref{fla42}.

\begin{lemma}\label{fla41}
Let $a,b,c,e,f\in P$ and $s,t,u,v,w,x,y,z\in L$ and 
suppose that $x^{\rm op}=y^{\rm op}$, that
$(b,x,a,y,c,u,b)$ and $\xi:=(b,w,e,t,c,u,b)$ are both closed $6$-paths
and that the $4$-paths $(y,c,t,e,z)$, $(y,c,v,f,s)$ and $(x,b,w,e,z)$ are all roots.
Then ${\rm dist}(b,f)=2$.
\end{lemma}

\begin{proof}
Note that $(a,y,c,v,f)$ is a root and hence $a$ and $f$ are opposite.
By \ref{toy20}, $a\not\in \Gamma_2(f)$
and hence $b\ne f$. By \ref{toy21}, therefore, it suffices to show that ${\rm dist}(b,f)\ne4$.
We suppose that ${\rm dist}(b,f)=4$. 
Since $a\in\Gamma_2(b)$, we have $\Gamma_2(b)\ne \Gamma_2(f)$. By \ref{fla5x}(ii),
it follows that $b$ and $f$ are opposite.
By \ref{fla10}(i), therefore, $\alpha:=(b,u,c,v,f)$ is a root.

Suppose now that $v=t$. Then $(u,c,v)$ lies on the root $\alpha$ and on the closed $6$-path
$\xi$. By \ref{fla40}, this is impossible. Hence $v\ne t$.
Hence by \ref{fla3}, we have $e\ne f$. Suppose that there exists $r\in\Gamma_e\cap\Gamma_f$.
By \ref{fla10}(i), $(b,w,e,r,f)$ is a root. Since $(x,b,w,e,z)$ is a root,
it follows that $(x,b,w,e,r)$ is also a root. Hence $r\in x^{\rm op}=y^{\rm op}$.
By \ref{fla10}(i), $\delta:=(r,f,v,c,y)$ is a root. 
By \ref{toy20}, $r\ne t$ and $r\ne v$. It follows that $\zeta:=(r,f,v,c,t,e,r)$ is a closed $6$-path.
The path $(r,f,v)$ is contained in both $\delta$ and $\zeta$. By \ref{fla40}
again, this is impossible. Hence the line $r$ does not exist and thus ${\rm dist}(e,f)=4$
by \ref{toy21}.

Since $b$ and $f$ are opposite, there exists a root $(b,w,g,p,f)$ from 
$b$ to $f$ that contains $w$. Since ${\rm dist}(e,f)=4$, we have $g\ne e$.
Since $(x,b,w,e,z)$ is a root, it follows that
$(x,b,w,g,p)$ is a root. Thus $p\in x^{\rm op}=y^{\rm op}$. By \ref{fla10}(i),
therefore, $(y,c,v,f,p)$ is a root. Hence
$(g,p,f,v,c)$ is a root. Thus $g$ and $c$ are opposite and therefore
the $\beta:=(g,w,e,t,c)$ is also a root, again by \ref{fla10}(i). 
The path $(w,e,t)$ thus lies on the root
$\beta$ and on the closed $6$-path $\xi$. By \ref{fla40},
this is impossible. With this last contradiction, we conclude that ${\rm dist}(b,f)\ne4$.
\end{proof}

\begin{proposition}\label{fla42}
Let $(x,a,y)$ be a path for some $a\in P$ and some $x,y\in L$ and suppose
that $x^{\rm op}=y^{\rm op}$. Then for each $c\in\Gamma_y\backslash\{a\}$,
there exists $b\in\Gamma_x\backslash\{a\}$ such that $(b,x,a,y,c)$ is a 
weed.
\end{proposition}

\begin{proof}
If $(x,a,y)$ is straight, then there exists a root starting at $x$ and containing
$a$ and $y$, but then the last vertex on this root is in $x^{\rm op}$ but not in 
$y^{\rm op}$. Hence $(x,a,y)$ is not straight.

Choose $c\in\Gamma_y\backslash\{a\}$.
Let $z\in y^{\rm op}$ and let $(y,c,t,e,z)$ be the root from $y$ to $z$ that
contains $c$. Then $(a,y,c,t,e)$ is also a root, so $a$ is opposite $e$.
Hence by \ref{fla40}, $a\not\in\Gamma_2(e)$.
Since $x^{\rm op}=y^{\rm op}$, 
there exists a unique root $(x,b,w,e,z)$ from $x$ to $z$ that contains $e$.
Since $b\in\Gamma_2(e)$, we have $a\ne b$. Thus by \ref{fla3}, $b\ne c$. 
We want to show that $(b,x,a,y,c)$ is a weed. By \ref{toy21}, it suffices to 
show that ${\rm dist}(b,c)\ne2$. We assume, therefore, that there exists
$u\in\Gamma_b\cap\Gamma_c$. 

By \ref{fla3}, $u\ne x$ and $u\ne y$. Hence $(a,x,b,u,c,y,a)$ is a closed $6$-path.
By \ref{fla40}, therefore, $(x,b,u)$ is not straight. Since $(x,b,w)$ is straight,
we have $u\ne w$. Since $(b,w,e)$ is contained in a root, it is straight and hence
$b\ne e$. We already know that $b\ne c$. It follows by \ref{fla3}, $u\ne t$ and $w\ne t$. Thus
$\xi=(b,w,e,t,c,u,b)$ is a closed $6$-path. 

Now suppose that $(y,c,v,f,s)$ is an arbitrary root starting with $y$ and then $c$.
Thus $s\in y^{\rm op}=x^{\rm op}$.
By \ref{fla41}, ${\rm dist}(b,f)=2$. By \ref{fla10}(i), therefore, $(x,b,r,f,s)$ is the unique
root from $x$ to $s$ that contains $f$, where $r=b\wedge f$. We conclude that 
$b$ does not depend on the choice of $z$. Thus
$u$ is also independent of the choice of $z$. Since
$\Gamma$ is $2$-plump, we can assume that $z$ is chosen so that $t$ is opposite $u$ at $c$.
Thus the $2$-path $(t,c,u)$ is both straight and contained in the $6$-circuit $\xi$. 
By \ref{fla40}, this is impossible. With this contradiction, we conclude that ${\rm dist}(b,c)\ne2$.
\end{proof}

\begin{proposition}\label{fla43}
Let $x,y\in L$. Then $x^{\rm op}=y^{\rm op}$ if and only if $x\approx y$.
\end{proposition}

\begin{proof}
Suppose that $x^{\rm op}=y^{\rm op}$. We claim that $x\approx y$. 
By \ref{fla42}, we have $x\sim y$ if ${\rm dist}(x,y)=2$.
By \ref{toy21}, therefore, it suffices to assume that ${\rm dist}(x,y)=4$. 
Let $(x,a,z,b,y)$ be a $4$-path from $x$ to $y$. 
By \ref{fla42} again, it will suffice to show that $x^{\rm op}=z^{\rm op}$ 
or $z^{\rm op}=y^{\rm op}$.
By \ref{fla11x}, we can assume that there is no weed containing $x$ and $z$
and no weed containing $z$ and $y$.
Since $x$ is not opposite to itself, it is not opposite to $y$. 
Since ${\rm dist}(x,y)=4$, it follows that $(x,a,z)$ or $(z,b,y)$ is not
straight. By symmetry, it suffices to consider the case that $(x,a,z)$ is
not straight. Since there is no weed containing $x$ and $z$, we thus have
\begin{equation}\label{fla43a}
\Gamma_x\subset \Gamma_2(b).
\end{equation}

Let $w\in x^{\rm op}$. There exists a root $(x,a,u,c,w)$ from $x$ to $w$
that contains $a$ and a root $(y,b,v,d,w)$ from $y$ to $w$ that contains $b$. 
Since $(x,a,u)$ is straight but $(x,a,z)$ is not, we have $u\ne z$ and hence $u\ne v$ 
by \ref{fla3}.
We claim that also $c\ne d$. Suppose, instead, 
that $c=d$. An element of $\Gamma_d$ is opposite $x$ (respectively $y)$ if and only if
it is opposite $u$ (respectively $v$) at $d$. Since $x^{\rm op}=y^{\rm op}$, it follows
that an element of $\Gamma_d$ is opposite $u$ at $d$ if and only if it is opposite $v$ at $d$.
Thus if $g,t$ are vertices such that $(u,d,w,g,t)$ is a $4$-path, then
$(u,d,w,g,t)$ is a root if and only if $(v,d,w,g,t)$ is a root. By \ref{fla10}(i), therefore,
$u^{\rm op}=v^{\rm op}$. Hence by \ref{fla42},
there exists $f\in\Gamma_u$ such that $(f,u,d,v,b)$ is a weed. In particular,
${\rm dist}(f,b)=4$, so $f\ne a$. By \ref{fla5x}(ii), we have $\Gamma_2(f)=\Gamma_2(b)$.
By \eqref{fla43a}, therefore, $\Gamma_x\subset \Gamma_2(f)$.
There thus exist $s,h$ such that $(x,a,u,f,s,h,x)$ is a closed $6$-path.
The path $(x,a,u)$ is contained in this path and in the root $(x,a,u,c,w)$.
By \ref{fla40}, this is impossible. With this contradiction, we conclude that, in fact, $c\ne d$,
as claimed. Thus $(c,w,d,v,b)$ is a root. 
By \ref{fla10}(i), therefore, $(c,u,a,z,b)$ is also a root.
Hence $(w,c,u,a,z)$ is a root and thus $w\in z^{\rm op}$. Therefore $x^{\rm op}\subset z^{\rm op}$.

Suppose, conversely, that $w\in z^{\rm op}$. By \ref{fla10}(i), there exist roots
$(z,a,u,c,w)$ and $(z,b,v,d,w)$ from $z$ to $w$, one containing $a$ and the
other containing $b$. By \ref{fla3}, $u\ne v$ and hence by \ref{fla40}, $c\ne d$. 
Therefore $(c,w,d,v,b)$ is a root. Thus $c$ and $b$ are opposite
and hence there is a root $(b,y,e,s,c)$ from $b$ to $c$ that contains $y$.
Therefore there exists $t\in\Gamma_c\backslash\{s\}$ such that $(t,c,s,e,y)$ is a root from $t$ to $y$.
Thus $t\in y^{\rm op}=x^{\rm op}$. By \ref{fla10}(i), therefore, $(t,c,u,a,x)$ is a root.
Thus $(u,a,x)$ and $(u,c,w)$ are both straight. It follows that $(x,a,u,c,w)$
is a root and thus $w\in x^{\rm op}$. We conclude that $x^{\rm op}=z^{\rm op}$.
Therefore $x^{\rm op}=y^{\rm op}$ implies $x\approx y$. The converse holds
by \ref{fla11x}.
\end{proof}

\begin{proposition}\label{fla32}
$\Gamma$ is flat if and only if $\Gamma$ contains no weeds.
\end{proposition}

\begin{proof}
If $\alpha=(a,x,b,y,c)$ is a weed in $\Gamma$, then 
$x^{\rm op}=y^{\rm op}$ by \ref{fla11x}. 
Hence if $\Gamma$ is flat, it contains no weeds. Suppose, instead,
that $\Gamma$ is not flat. If there exist distinct points
$a,b$ such that $a^{\rm op}=b^{\rm op}$, then 
by \ref{fla5x}(i), $a\simeq b$ and hence there is a weed from
$a$ to $b$. Suppose that there exist distinct lines $x,y$ such that $x^{\rm op}=y^{\rm op}$.
By \ref{fla43}, $x\approx y$. Hence there exists a line $z$ such that $x\sim z$.
It follows that there is a weed containing $x$ and $z$.
\end{proof}

\begin{proposition}\label{fla4}
Let $x,y$ be elements of $L$ such that ${\rm dist}(x,y)=2$ and $x^{\rm op}=y^{\rm op}$ 
and let $b=x\wedge y$. Then 
there exists a unique map $\varphi_{xy}$ from $\Gamma_x$ to $\Gamma_y$ mapping $b$ to itself
such that for all $a\in\Gamma_x\backslash\{b\}$, $(a,x,b,y,\varphi_{xy}(a))$ is a weed. 
\end{proposition}

\begin{proof}
Let $a\in\Gamma_x\backslash\{b\}$. By \ref{fla42}, there exists $c$ such that 
$(a,x,b,y,c)$ is a weed.
If $c'\in\Gamma_y\backslash\{b,c\}$, then $c'\in \Gamma_2(c)=\Gamma_2(a)$ by \ref{fla5x}(ii) and hence 
$(a,x,b,y,c')$ is not a weed. 
\end{proof}

\begin{remark}\label{fla4x}
Let $x,y$ be as in \ref{fla4}. Applying \ref{fla4} to $x,y$ and to $y,x$, we obtain
maps $\varphi_{xy}$ and $\varphi_{yx}$. These two maps are inverses of each
other and hence both are bijections.
\end{remark}

\begin{remark}\label{fla4y}
By \ref{fla11x}, the map $\varphi_{xy}$ exists for all $x,y\in L$ such that $x\sim y$.
\end{remark}

\begin{proposition}\label{fla8}
Let $a\in P$. Then the restriction of the reflexive closure of $\sim$ to $\Gamma_a$ is an
equivalence relation.
\end{proposition}

\begin{proof}
Suppose that $u\sim v$, $v\sim w$ and $w\ne u$ for $u,v,w\in\Gamma_a$.
By \ref{fla4y}, there exist $b\in\Gamma_u$, $c\in\Gamma_v$ and $d\in\Gamma_w$ such that $b\simeq c$
and $c\simeq d$. By \ref{fla70},
it follows that $b\simeq d$. By \ref{fla3}, $b\ne d$. 
Hence by \ref{fla10}(ii), $(b,u,a,w,d)$ is a weed. Hence $u\sim w$.
\end{proof}

\begin{proposition}\label{fla7}
Let $x,y\in\Gamma_a$ for some $a\in P$ and suppose that ${\rm dist}(x,y)=2$. Then $x\sim y$ if and
only if $x\approx y$.
\end{proposition}

\begin{proof}
Let $u_0,u_1,u_2\in L$ and suppose that $u_0\in\Gamma_a$, $u_0\sim u_1$,
$u_1\sim u_2$ and $u_2\ne u_0$. By \ref{fla4y},
we can choose a weed $(a,u_0,b,u_1,c)$ starting at $a$ and containing $u_0$ and $u_1$.

We claim that the following hold:
\begin{enumerate}[\rm(i)]
\item If $u_2\in\Gamma_a$, then $u_0\sim u_2$.
\item If $u_2\not\in\Gamma_a$, then there exists $u_2'\in\Gamma_a$ such that 
$u_2'\sim u_2$ and either $u_2'\sim u_0$ or $u_2'=u_0$.
\end{enumerate}
Choose $h\in\Gamma_{u_0}\backslash\{a,b\}$.
Let $i=\varphi_{u_0u_1}(h)$ and $j=\varphi_{u_1u_2}(i)$. Then $h\simeq i$ and $i\simeq j$
so by \ref{fla5}, $h\simeq j$. Since ${\rm dist}(a,h)=2$, we have $a\not\simeq h$
and thus $j\ne a$. We also have $j\in\Gamma_{u_2}$. Suppose now that $u_2\in\Gamma_a$. 
Since $u_0\ne u_2$, $(h,u_0,a,u_2,j)$ is a $4$-path and therefore 
a weed by \ref{fla10}(ii). Thus (i) holds. Suppose, instead, that $u_2\not\in\Gamma_a$. 
In this case, we let $g=\varphi_{u_1u_2}(c)$. 
Then $g\simeq c$. Since $a\simeq c$, it follows by \ref{fla70} that $a\simeq g$. 
By \ref{fla10}(ii), therefore, there exists a weed $(a,u_2',f,u_2,g)$ from $a$ to $g$ containing $u_2$. 
Thus $u_2'\sim u_2$.  
Let $k=\varphi_{u_2u_2'}(j)$. Then $j\simeq k$ and hence $h\simeq k$. Since
$a\not\simeq h$, we have $k\ne a$. We also have $k\in\Gamma_{u_2'}$.
Hence if $u_0\ne u_2'$, then $(h,u_0,a,u_2',k)$ is a 
$4$-path and thus a weed, again by \ref{fla10}(ii), and therefore $u_0\sim u_2'$.
Thus (ii) holds.

Let $\Omega$ be the set of sequences $u_0,\ldots,u_m$ of elements of $L$ for arbitrary $m$
such that $u_i\sim u_{i-1}$ for all $i\in[1,m]$ and $u_0,u_m\in\Gamma_a$ but neither
$u_0=u_m$ nor $u_0\sim u_m$ holds. Suppose that $\Omega$ is non-empty. Let 
$u_0,\ldots,u_m$ be a sequence in $\Omega$ of minimal length. Then $m\ge2$. Suppose $u_2\in\Gamma_a$.
The sequence $u_2,\dots,u_m$ is too short to be in $\Omega$. Hence $u_2\sim u_m$ or $u_2=u_m$.
By (i) and \ref{fla8}, it follows that $u_0\sim u_m$.
This contradicts, however, our choice of $u_0,\ldots,u_m$. Suppose instead that
$u_2\not\in\Gamma_a$ and let $u_2'$ be as in (ii). Thus $u_2'\sim u_2$, but 
again the sequence $u_2',u_2,u_3,\ldots,u_m$ is too short to 
be in $\Omega$. Hence $u_2'\sim u_m$ or $u_2'=u_m$. Since either $u_0=u_2'$ or $u_0\sim u_2'$,
we conclude that $u_0=u_m$ or $u_0\sim u_m$ by one more
application of \ref{fla8}. Again this contradicts our choice of $u_0,\ldots,u_m$. 
Thus $\Omega$ is empty.
\end{proof}

\begin{notation}\label{fla12}
Let $\bar P$ be the set of equivalence classes of the equivalence relation $\simeq$ on $P$
and let $\bar L$ be the set of equivalence classes of the equivalence relation $\approx$ on $L$.
Let $\bar\Gamma$ be the bipartite graph whose vertex set is the disjoint union of $\bar P$
and $\bar L$, where
$[a]\in\bar P$ is adjacent to $[x]\in\bar P$ whenever there exists an element of $[a]$
adjacent in $\Gamma$ to an element of $[x]$. We declare that two elements of $\bar P$
adjacent in $\bar\Gamma$
to a vertex $[x]$ in $\bar L$ are {\it opposite at $[x]$} if they are distinct
and we declare that two elements $[x]$ and $[y]$ of $\bar L$
adjacent in $\bar\Gamma$ to a vertex $[a]$ in
$\bar P$ are {\it opposite at $[a]$}
if there is point $b$ in $[a]$ such that all the lines in $\Gamma_b\cap[x]$
are all opposite all the lines in $\Gamma_b\cap[y]$. 
We can identify $\bar\Gamma$ with $\Gamma$ in the case that
the equivalence classes $[x]$ for $x\in L$ and 
$[a]$ for $a\in P$ are all singletons, 
\end{notation}

\begin{proposition}\label{fla32x}
$\Gamma$ is flat if and only $\Gamma=\bar\Gamma$.
\end{proposition}

\begin{proof}
We have $[x]\sim[y]$ for distinct
$x,y\in L$ if and only if there is a weed containing $x$ and $y$ and $[a]\simeq[b]$ for 
distinct $a,b\in P$ if and only if there is a weed containing $a$ and $b$. Hence
$\Gamma$ contains no weeds if and only if the equivalence classes comprising the 
vertex set of $\bar\Gamma$ are all singletons. Hence $\Gamma$ contains no weeds if and only 
if $\Gamma=\bar\Gamma$. The claim holds, therefore, by \ref{fla32}.
\end{proof}

\begin{theorem}\label{fla14}
Let $\bar\Gamma$ be the graph endowed with local opposition relations defined in {\rm\ref{fla12}}.
Then $\bar\Gamma$ is a flat green Veldkamp quadrangle.
\end{theorem}

\begin{proof}
We will prove \ref{fla14} in a series of steps. 

\begin{proposition}\label{fla15}
Let $[a]\in\bar P$ and $[x]\in\bar L$ be adjacent in $\bar\Gamma$.
Then $\Gamma_b\cap[x]\ne\emptyset$ for each $b\in[a]$ and
$\Gamma_y\cap[a]\ne\emptyset$ for each $y\in[x]$.
\end{proposition}

\begin{proof}
Since $[a]$ and $[x]$ are adjacent in $\bar\Gamma$, we can assume
that $x\in\Gamma_a$. Let $b\in[a]\backslash\{a\}$. 
Then $a\simeq b$ but $a\ne b$. By \ref{fla10}(ii), there exists
a weed $(a,x,c,y,b)$ from $a$ to $b$ 
that contains $x$. Then $y\in\Gamma_b\cap[x]$.

Now let $z\in[x]\backslash\{x\}$. We want to show that $\Gamma_z\cap[a]\ne\emptyset$.
It will suffice to assume that $z\sim x$ (rather than $z\approx x$) and that
$z\not\in\Gamma_z$. By 
\ref{fla11x} and \ref{fla42}, there is a weed starting at $a$ and containing $x$ and $z$.
Thus $e\in\Gamma_z\cap[a]$, where $e$ is the last vertex of this weed.
\end{proof}

\begin{proposition}\label{fla16}
Let $a,b\in P$, $x,y\in\Gamma_a$ and $u,v\in\Gamma_b$. Suppose that $x$ and $y$
are opposite at $a$, $[a]=[b]$, $[x]=[u]$ and $[y]=[v]$. Then $u$ and $v$ are opposite
at $b$.
\end{proposition}

\begin{proof}
Suppose that $a=b$. By \ref{fla7}, we have $x\sim u$ or $x=u$ and $y\sim v$ or $y=v$.
By \ref{fla11}, therefore, $u$ and $v$ are opposite at $a$.

Suppose that $a\ne b$. 
By \ref{fla10}(ii), there exist weeds
$\alpha=(a,x,d,w,b)$ and $\beta=(a,y,e,z,b)$ 
from $a$ to $b$ passing through $x$ and through $y$.
Then $[u]=[x]=[w]$ and $[v]=[y]=[z]$. Since $x$ and $y$ are opposite at $a$,
$(d,x,a,y,e)$ is a root. By \ref{fla10}(i), therefore, $(d,w,b,z,e)$ is a root. Hence
$w$ and $z$ are opposite at $b$. By the conclusion of the previous paragraph
(with $b$ in place of $a$), it follows that $u$ and $v$ are opposite at $b$.
\end{proof}

\begin{proposition}\label{fla20}
Let $\pi$ be the natural map from $\Gamma$ onto $\bar\Gamma$ sending
$a$ to $[a]$ for all $a\in P$ and $x$ to $[x]$ for all $x\in L$.
Then the following hold:
\begin{enumerate}[\rm(i)]
\item $\pi$ maps edges of $\Gamma$ to edges of $\bar\Gamma$.
\item For each vertex $u$ of $\Gamma$, paths in $\bar\Gamma$ starting at $[u]$ 
lift via $\pi$ to paths in $\Gamma$ starting at $u$.
\item Let $\gamma$ be a path in $\Gamma$. Then $\gamma$ is straight if and only if
$\pi(\gamma)$ is a straight path in $\bar\Gamma$.
\end{enumerate}
\end{proposition}

\begin{proof}
The assertions follow from \ref{fla12}, \ref{fla15} and \ref{fla16}. We only need to observe that
if $(a,x,b)$ is a $2$-path with $x\in L$, then ${\rm dist}(a,b)=2$ and therefore
$[a]\ne[b]$ and if $(x,a,y)$ is a straight $2$-path with $a\in P$, then there is no weed containing
$x$ and $y$ and hence $[x]\ne[y]$ by \ref{fla7}.
\end{proof}

\begin{proposition}\label{fla21}
$\bar\Gamma$ has no circuits of length~$4$.
\end{proposition}

\begin{proof}
Suppose that $\beta:=([a],[x],[b],[y],[c])$ is a $4$-path in $\bar\Gamma$
such that $[a]=[c]$. By \ref{fla20}(ii), we can assume that $\alpha:=(a,x,b,y,c)$ is a $4$-path
in $\Gamma$. By \ref{fla3}, $a\ne c$. By \ref{fla10}(ii), therefore, $\alpha$ is a weed.
Therefore $[x]=[y]$. This contradicts the assumption that $\beta$ is a path.
\end{proof}

\begin{proposition}\label{fla22}
Let $\alpha=([a],[x],[b],[y])$ be a straight $3$-path in $\bar\Gamma$ with 
$a\in P$. Then $\alpha$ is the unique straight path of length at most~$3$ in $\bar\Gamma$
from $[a]$ to $[y]$.
\end{proposition}

\begin{proof}
By \ref{fla20}(ii)-(iii), we can assume that $(a,x,b,y)$ is a straight path in $\Gamma$. 
By \ref{fla21}, $[a]$ and $[y]$ are not adjacent.
Since $\bar\Gamma$ is bipartite, it follows that the distance from $[a]$ to $[y]$ in $\bar\Gamma$
is~$3$. Let $\gamma=([a],[z],[d],[y])$ be an arbitrary straight $3$-path from $[a]$ to $[y]$.
It will suffice to show that $\gamma=\alpha$.

By \ref{fla20}(ii)-(iii), we can assume that $(e,z,d,y)$ is a straight $3$-path for
some $e\in\Gamma_z\cap[a]$. By \ref{fla5x}(ii), $\Gamma_2(e)=\Gamma_2(a)$.
Hence $b\in \Gamma_2(e)$, so by \ref{toy20}, $b$ and $e$ are not opposite.
Therefore $(e,z,d,y,b)$ is not a root. Hence
$b=d$. Thus $([a],[z],[d])$ and $([a],[x],[b])$ are both paths
in $\bar\Gamma$ from $[a]$ to $[d]$. By \ref{fla21}, it follows that $[z]=[x]$.
Therefore $\gamma=\alpha$.
\end{proof}

We can now show that $\bar\Gamma$ satisfies the axioms of a Veldkamp quadrangle.
Let $u\in P\cup L$ and let $[v]$ and $[w]$ be two vertices adjacent to $[u]$ in $\bar\Gamma$.
By \ref{fla20}(ii), we can assume that $v,w\in\Gamma_u$. Since $\Gamma$ is $2$-plump,
there exists $z\in\Gamma_u$
opposite $u$ and $v$ at $u$. By \ref{fla20}(iii), $[z]$ is opposite $[u]$ and $[v]$ at $[u]$.
Thus $\bar\Gamma$ is $2$-plump.

Since $\Gamma$ is connected and every vertex of $\bar\Gamma$ is in the image of $\pi$, 
\ref{fla20}(i) implies that $\bar\Gamma$ is connected. It
is bipartite by construction. Thus $\bar\Gamma$ satisfies (VP1).
By \ref{fla21} and \ref{fla22}, $\bar\Gamma$ satisfies (VP2) with $n=4$. 
By \ref{fla20}(ii)-(iii), a straight $5$-path in $\bar\Gamma$ lifts to a straight $5$-path
in $\Gamma$ (via $\pi$) and a straight $8$-circuit in $\Gamma$ maps to a straight
$8$-circuit in $\bar\Gamma$. Since $\Gamma$ satisfies (VP3) with $n=4$,
it follows that $\bar\Gamma$ does as well. 
By \ref{fla12}, the local opposition relations at elements of $\bar L$ are trivial.
Thus $\bar\Gamma$ is a green Veldkamp quadrangle.

Suppose that $([a],[x],[b],[y],[c])$ is a weed in $\bar\Gamma$. 
Thus the distance between $[a]$ and $[c]$ in $\bar\Gamma$ is~$4$. 
By \ref{fla20}(ii)-(iii), 
we can assume that $\alpha:=(a,x,b,y,c)$ is a path that is not straight
and by \ref{fla20}(i), ${\rm dist}(a,c)=4$.
Hence $\alpha$ is a weed, but this implies that $[a]=[c]$. With this contradiction,
we conclude that $\bar\Gamma$ does not contain
any weeds. By \ref{fla32}, it follows that $\bar\Gamma$ is flat. 
This concludes the proof of \ref{fla14}.
\end{proof}

\begin{proposition}\label{fla30}
If $\Gamma$ is $k$-plump for some $k\ge3$, then $\bar\Gamma$ is also $k$-plump.
\end{proposition}

\begin{proof}
This holds by the same argument just used to show that $\bar\Gamma$ $2$-plump.
\end{proof}

\begin{remark}\label{fla88}
Note that the construction of $\bar\Gamma$
in \ref{fla12} does not involve any choices and thus $\bar\Gamma$ is uniquely determined
by $\Gamma$. We call it the {\it flat quotient} of $\Gamma$.
\end{remark}

\begin{proposition}\label{fla31}
The following hold:
\begin{enumerate}[\rm(i)]
\item $\pi$ restricts to an bijection from $\Gamma_x$ to $\bar\Gamma_{[x]}$ for all $x\in L$.
\item $\pi(\Gamma_a)=\bar\Gamma_{[a]}$ for all $a\in P$, but the restriction of $\pi$ to $\Gamma_a$ is 
injective for all $a\in P$ if and only if $\Gamma$ is flat.
\end{enumerate}
\end{proposition}

\begin{proof}
By \ref{fla12}, $\pi(\Gamma_x)\subset\bar\Gamma_{[x]}$ and 
$\pi(\Gamma_a)\subset\bar\Gamma_{[a]}$ for all $x\in L$ and all $a\in P$. 
By \ref{fla20}(ii), both of these inclusions are, in fact, equalities and
by \ref{fla20}(iii), therefore, (i) holds.
Let $a\in P$ and suppose that $[y]=[z]$ for some $y,z\in\Gamma_a$. 
By \ref{fla7}, either $y\sim z$ or $y=z$. 
It follows that the restriction of $\pi$ to $\Gamma_{[a]}$ is injective
if and only if there no weeds whose middle point is $a$. Hence the restriction
of $\pi$ to $\Gamma_{[a]}$ is injective for all $a\in P$ if and only if $\Gamma$ contains
no weeds. Thus (ii) holds by \ref{fla32}.
\end{proof}

\section{Polar spaces}

In this section we describe the connection between polar spaces and Veldkamp quadrangles.
The main results are \ref{hau12} and \ref{hau13}. The proofs require
a few basic properties of polar spaces for which we will cite results from \cite[Chapter~7]{BC}.

\begin{definition}\label{hau1}
A {\it line space} is a pair $S=(P,L)$ consisting of a set $P$ of points and a subset $L$
of $2^P$ such that $|x|\ge2$ for all $x\in L$.
The elements of $L$ are called {\it lines}.
A line $x$ is {\it thick} if $|x|\ge3$ and $S$ is called {\it thick} if all its lines are thick. 
Two points are called {\it collinear} if there is an element of $L$ that contains them both 
\end{definition}

\begin{notation}\label{hau1x}
Let $S=(P,L)$ be a line space. The associated biparite graph $\Gamma=\Gamma_S$ is
the graph with vertex set $P\cup L$ and edge set consisting of all pairs $\{a,x\}$
such that $a\in P$, $x\in L$ and $a\in x$.
\end{notation}

\begin{notation}\label{hau2}
Let $S=(P,L)$ be a line space. For each $a\in L$, 
we denote by $a^\perp$ the set of all points collinear with $a$ (including $a$ itself)
and we set 
$$X^\perp=\bigcap_{a\in X}a^\perp$$
for all subsets $X$ of $P$. A {\it subspace} is a subset $X$ of $P$ such that for all
$x\in L$, either $x\subset X$ or $|x\cap X|\le1$. We will often identify a subspace
$X$ with the line space
$$S_X:=(X,\{x\in L\mid x\subset X\}).$$
For all subsets $X$ of $P$, let $\langle X\rangle$ denote the intersection of all
the subspaces that contain $X$. This intersection is itself a subspace. It is called the subspace
{\it generated by $X$}. 
\end{notation}

\begin{notation}\label{hau3}
A {\it linear space} is a line space such that for every two points there is a unique
line containing them both. A {\it partially linear space} is a line space such that for 
every two points there is at most one line containing them both. If 
$a,b$ are distinct collinear points in a partially linear space, then the unique line containing them 
both is denoted by $ab$. 
\end{notation}

\begin{definition}\label{hau4}
A {\it projective plane} is a linear space $S=(P,L)$ such that $|L|\ge2$ and 
$x\cap y\ne\emptyset$ for all $x,y\in P$. A {\it projective space} is a linear space
such that every subspace generated by the union of two lines with non-empty intersection is a
projective plane.
\end{definition}

\begin{definition}\label{hau5}
Let $S=(P,L)$ be a projective space. The {\it dimension} of $S$ is the maximal
length~$k$ of a properly ascending chain of subspaces $X_i$ such that
$$\emptyset\ne X_1\subset X_2\subset\cdots\subset X_k=P.$$
It is denoted by ${\rm dim}(S)$. Thus, in particular, lines are 
subspaces of dimension~$1$ and ${\rm dim}(S)=\infty$ if there no bound on the 
length of such chains. A non-singular subspace of dimension~$2$ is called
a non-singular {\it plane}.
\end{definition}

\begin{definition}\label{hau7}
A {\it polar space} is a line space $S=(P,L)$ such that the following holds:

\medskip\noindent
(BS) For all $a\in P$ and $x\in L$, either $|a^\perp\cap x|=1$ or $x\subset a^\perp$.

\medskip\noindent
Note that every subspace of a polar space is itself a polar space.
\end{definition}

\begin{definition}\label{hau8}
Let $S=(P,L)$ be a line space. A subspace $X$ of $S$ is called {\it singular} if
every two points of $X$ are collinear. Thus lines are, in particular, singular
subpaces.
\end{definition}

\begin{definition}\label{hau9}
A polar space $S=(P,L)$ is {\it non-degenerate} if $P\ne a^\perp$ for all $a\in P$.
\end{definition}

\begin{proposition}\label{hau11}
Let $S$ be a non-degenerate polar space. Then the following hold:
\begin{enumerate}[\rm(i)]
\item $S$ is a partially linear space.
\item Singular subspaces are projective spaces.
\end{enumerate}
\end{proposition}

\begin{proof}
(i) holds by \cite[Thm.~7.4.11]{BC} and (ii) by \cite[Thm.~7.4.13(iv)]{BC}.
\end{proof}

\begin{definition}\label{hau9x}
In light of \ref{hau11}(ii), we can define the
{\it rank} of a non-degenerate polar space $S=(P,L)$ to be one more than the supremum of the set
$$\{{\rm dim}(E)\mid\text{$E$ is a non-singular subspace}\}.$$
The rank is denoted by ${\rm rk}(S)$. Thus ${\rm rk}(S)=1$ if and only if $L=\emptyset$
but $P\ne\emptyset$,
and  ${\rm rk}(S)=2$ if and only if $L\ne\emptyset$ but lines are maximal singular
subspaces.
\end{definition}

\begin{remark}\label{hau505}
A polar space is called thick in this paper (see \ref{hau1}) if all its lines are thick. 
We note that the definition of a thick polar space in \cite[7.3]{spherical} is slightly different.
\end{remark}

\begin{proposition}\label{hau21}
Let $S=(P,L)$ be a thick non-degenerate polar space and let $a,b\in P$.
Then there exists a point that is collinear with neither $a$ nor $b$.
\end{proposition}

\begin{proof}
Suppose first that $a$ and $b$ are collinear and let $x=ab$. Since $S$ is thick,
we can choose $c\in x$ distinct from $a$ and $b$. By \cite[7.4.8(iii)]{BC},
there exists $d\in c^\perp\backslash x^\perp$. This point $d$ is collinear with 
neither $a$ nor $b$. Now suppose that $a$ and $b$ are not collinear
Let $x$ be a line containing $b$ and let $c$ be a point on $x$ distinct from $b$.
By (BS), $a$ is collinear with exactly one point on $x$. Since $S$ is thick,
we can assume, therefore, that $a$ and $c$ are not collinear. By another application 
of \cite[7.4.8(iii)]{BC}, there exists $d\in c^\perp\backslash x^\perp$. Let $y=cd$.
By (BS), there is a unique point $e$ in $y\cap a^\perp$. Since $S$ is thick,
we can choose $f$ on $y$ distinct from $c$ and $e$.  Thus $f$ is not collinear with $a$,
nor, by the choice of $b$, is it collinear with $b$.
\end{proof}

\begin{proposition}\label{hau22}
Let $S=(P,L)$ be a thick non-degenerate polar space and let $a$ and $b$
be two non-collinear points. Then $a^\perp\cap b^\perp$ is 
a thick non-degenerate polar space.
\end{proposition}

\begin{proof}
Let $S_1=a^\perp\cap b^\perp$. If $x$ is a line containing two points $d,e$ of $S_1$,
then by (BS), $x\subset S_1$. Thus $S_1$ is a thick subspace. In particular, 
$S_1$ is a polar space. By \cite[7.4.8(ii)]{BC}, it is non-degenerate.
\end{proof}

\begin{proposition}\label{hau22x}
Let $S=(P,L)$ be a thick non-degenerate polar space and let $x\in L$.
Then there exists $y\in L$ such that $x^\perp\cap y^\perp$ is 
a thick non-degenerate polar space and $x\cap y^\perp=\emptyset$.
\end{proposition}

\begin{proof}
Let $a\in x$. Since $S$ is non-degenerate, we can choose $e\in P\backslash a^\perp$
and by (BS), we can choose $b\in x$ collinear with $e$. 
Let $S_1=a^\perp\cap e^\perp$. 
By \ref{hau22}, $S_1$ is a thick non-degenerate polar space. In particular, we
can choose $d\in S_1$ not collinear with $b$. Let $y=de$ and let 
$S_2=b^\perp\cap d^\perp\cap S_1$. Thus
$$S_2=a^\perp\cap b^\perp\cap d^\perp\cap e^\perp.$$
By (BS), $a^\perp\cap b^\perp=x^\perp$ and
$d^\perp\cap e^\perp=y^\perp$. Hence $S_2=x^\perp\cap y^\perp$.
By a second application of \ref{hau22},
$S_2$ is a thick non-degenerate polar space. 

Now suppose that $c\in x\cap y^\perp$. Then $b,c\in e^\perp$ as well as $c\in d^\perp$, so $c\ne b$.
By (BS), therefore, $a\in e^\perp$. This contradicts the
choice of $e$. We conclude that $x\cap y^\perp=\emptyset$.
\end{proof}

\begin{proposition}\label{hau22y}
Let $x,y\in L$ be as in {\rm\ref{hau22x}}, let $S_2=x^\perp\cap y^\perp$
and let $E$ be a singular plane containing $x$.
Then $|E\cap S_2|=1$.
\end{proposition}

\begin{proof}
By \ref{hau4} and \ref{hau11}(ii), $E$ is a projective plane. Thus
every two lines contained in $E$ intersect in a point. In particular, $E\subset x^\perp$.
Let $d\in y$. By (BS), $d^\perp$ contains a point on each line of $E$.
Hence $|d^\perp\cap E|\ge2$. It follows that $E\cap d^\perp$
contains a line. Now choose a second point $e$ in $y$. Then $E\cap e^\perp$
also contains a line. Therefore $E\cap d^\perp\cap e^\perp\ne\emptyset$.
By (BS), we have $d^\perp\cap e^\perp=y^\perp$. Thus $E\cap S_2=E\cap y^\perp\ne\emptyset$.

Suppose that $a,b$ are two points in $E\cap S_2$. Then $a$ and $b$ are collinear. 
Let $z=ab$. Then $z\subset E\cap S_2$. It follows that $z$ and $x$ intersect in a point
that lies in $x\cap S_2\subset x\cap y^\perp$. By \ref{hau22x}, however, $x\cap y^\perp=\emptyset$.
With this contradiction, we conclude that $|E\cap S_2|\le1$.
\end{proof}

\begin{proposition}\label{hau22z}
Let $x\in L$ and let $M$ be the set of singular planes containing $x$.
Then for all $E,F\in M$ there exists an element $G$ of $M$ such that both
$\langle E\cup G\rangle$ and $\langle F\cup G\rangle$ are non-singular.
\end{proposition}

\begin{proof}
Let $y\in L$ be as in {\rm\ref{hau22x}} and let $S_2=x^\perp\cap y^\perp$.
Thus $S_2$ is a thick non-degenerate polar space and $x\cap y^\perp=\emptyset$.
Choose $E,F\in M$. By \ref{hau22y}, there exist points $a,b$ such that
$E\cap S_2=\{a\}$ and $F\cap S_2=\{b\}$.
By \ref{hau21}, there exists $c\in S_2$ that is collinear with neither $a$ nor $b$.
Let $G=\langle x\cup c\rangle$. Since $x\cap y^\perp=\emptyset$, we have 
$c\not\in x$. On the other hand, $c\in x^\perp$ so by \cite[7.4.7(v)]{BC}, 
$G$ is singular. Thus $G\in M$.
By the choice of $c$, neither $\langle E\cup G\rangle$ nor $\langle F\cup G\rangle$ is
singular.
\end{proof}

\begin{proposition}\label{hau12}
Let $S=(P,L)$ be a thick, non-degenerate polar space such that $L\ne\emptyset$
and let $\Gamma=\Gamma_S$ be the associated bipartite graph as defined in 
{\rm\ref{hau1x}}. Let $\Gamma$
be endowed with the following local opposition relations:
\begin{enumerate}[\rm(a)]
\item For each $a\in P$ and all $x,y\in\Gamma_a$, $x$ is opposite $y$ at $a$ 
if the subspace $\langle x\cup y\rangle$ is non-singular.
\item For each $x\in L$ and all $a,b\in\Gamma_x$, $a$ is opposite $b$ at $x$ 
if $a\ne b$. 
\end{enumerate}
Then $\Gamma$ is a flat green Veldkamp quadrangle. 
It is a generalized quadrangle if and only if ${\rm rk}(S)=2$.
\end{proposition}

\begin{proof}
Let $a\in P$ and $x,y\in\Gamma_a$. Then \ref{hau4} and \ref{hau11}(ii), imply that
if $x$ and $y$ are not opposite at $a$, then 
every point in $x\backslash\{a\}$ is 
collinear with every point in $y\in\backslash\{a\}$. By \cite[7.4.7(v)]{BC}, the 
converse holds. Thus $x$ and $y$ are opposite at $a$ if and only if 
some point in $x\backslash\{a\}$ is not collinear with some point in $y\in\backslash\{a\}$.
By (BS), it follows that if $x$ and $y$ are opposite at $a$, then 
no point in $x\backslash\{a\}$ is collinear with any point in $y\in\backslash\{a\}$.

We claim that $\Gamma$ is $2$-plump. Let $a\in P$. Since $S$ is non-degenerate, we can 
choose a point $b$ not collinear with $a$. Let $x,y\in\Gamma_a$. By (BS), there is a 
unique point $d$ on $x$ collinear with $b$ and a unique point $e$ on $y$ collinear with $b$.
By \ref{hau21} and \ref{hau22}, there exists a point $f$ in $a^\perp\cap b^\perp$
that is collinear with neither $d$ nor $e$. Let $z=af$. Then $z\in\Gamma_a$ is locally 
opposite $x$ and $y$ at $a$. Since $S$ is thick, $|\Gamma_x|\ge3$ also for $x\in L$.
Thus $\Gamma$ is $2$-plump, as claimed.

By (BS), ${\rm dist}(a,x)\le3$ for all $a\in P$ and all $x\in L$. Thus 
\begin{equation}\label{hau12b}
{\rm diam}(\Gamma)\le4.
\end{equation}
In particular, $\Gamma$ is connected. Thus it satisfies (VP1). 
It follows from \ref{hau11}(i) that $\Gamma$ does not contain
any $4$-circuits. Let $(a,x,b,y)$ is a $3$-path with $a\in P$. Suppose that $\langle x\cup y\rangle$
is non-singular. Then every point in $x$ is collinear with $b$ but no other
point in $y$. Hence ${\rm dist}(a,c)=4$ for all $c\in y\backslash\{b\}$. In other words, 
$(a,x,b,y)$ does not lie on a $6$-circuit. Thus $\Gamma$ satisfies (VP2).
Now suppose that $\langle x\cup y\rangle$ is singular. Then $a$ is collinear with every point on $y$.
Hence $\Gamma_y\subset \Gamma_2(a)$. Thus we conclude that 
\begin{equation}\label{hau12c}
\Gamma_w\subset \Gamma_2(f)
\end{equation}
for all $w\in L$ and all $f\in P$ such that there is a non-straight $3$-path from $w$ to $f$
and hence
\begin{equation}\label{hau12a}
e^{\rm op}=\Gamma_4(e)
\end{equation}
for all points $e$.

Now suppose that $(a,x,b,y,c,z)$ be a straight $5$-path. 
Then ${\rm dist}(a,c)=4$ by \eqref{hau12a}, so by \eqref{hau12b}, 
we have ${\rm dist}(a,z)=3$. Hence there exist $d,u$ such that
$(a,x,b,y,c,z,d,u,a)$ is a closed $8$-path. The subspace $\langle z\cup u\rangle$ is non-singular
since otherwise $c\in\Gamma_z\subset \Gamma_2(a)$ by \eqref{hau12c}. Since the
subspace $\langle y\cup z\rangle$ is also non-singular, we have
${\rm dist}(b,d)=4$ by \eqref{hau12a}. Thus $\langle x\cup u\rangle$ is non-singular
since otherwise $d\in\Gamma_u\subset \Gamma_2(b)$ by \eqref{hau12c}.
We conclude that $\Gamma$ satisfies (VP3). Thus $\Gamma$ is a green Veldkamp quadrangle.
Suppose that $a^{\rm op}=b^{\rm op}$ for $a,b\in P$. 
Since $b$ is not opposite itself, it follows by \eqref{hau12a} that ${\rm dist}(a,b)\ne4$.
Suppose that ${\rm dist}(a,b)=2$. Since $\Gamma$ is green, there exists a 
root $(a,w,b,v,c)$ starting at $a$ and containing $b$, but then the last 
vertex $c$ of this root would be in $a^{\rm op}\backslash b^{\rm op}$. 
By \eqref{hau12b}, therefore, $a=b$. By \ref{fla5x}(i), it follows that $\Gamma$ 
contains no weeds. By \ref{fla32}, we conclude that $\Gamma$ is flat. 

If ${\rm rk}(S)=2$, then $\langle x\cup y\rangle$ is non-singular for all $x,y\in L$,
hence every path in $\Gamma$ is straight and thus $\Gamma$ is a generalized $4$-gon.
Suppose, conversely, that every path in $\Gamma$ is straight. Then 
$\langle x\cup y\rangle$ is non-singular for all $x,y\in L$ such that $x\cap y\ne\emptyset$.
By \ref{hau4} and \ref{hau11}(ii), $x\cap y\ne\emptyset$ for every two lines contained
in a singular plane. Hence there are no singular planes. In other words, ${\rm rk}(S)=2$.
\end{proof}

\begin{theorem}\label{hau13}
Let $\Gamma=(V,E)$ be a flat green Veldkamp quadrangle with point set $P$ and line set $L$
and let $L_1=\{\Gamma_x\mid x\in L\}$.
Then the following hold:
\begin{enumerate}[\rm(i)]
\item $S_\Gamma:=(P,L_1)$ is a thick non-degenerate polar space with $L_1\ne\emptyset$.
\item The construction $\Gamma\rightsquigarrow S_\Gamma$ and the construction
$S\rightsquigarrow\Gamma_S$ in {\rm\ref{hau12}}
are inverses of each other.
\item ${\rm rk}(S_\Gamma)=2$ if and only if $\Gamma$ is a generalized quadrangle.
\end{enumerate}
\end{theorem}

\begin{proof}
Let $a\in P$ and let $x\in L$. 
By \ref{toy21}, ${\rm dist}(a,x)\le3$. Suppose that
${\rm dist}(a,x)=3$ but $\Gamma_x\not\subset \Gamma_2(a)$. Then there exists a 
$4$-path $\alpha=(a,u,b,x,c)$ such that ${\rm dist}(a,c)=4$. 
By \ref{fla32}, $\Gamma$ contains no weeds. It follows that $\alpha$ is a root.
In particular, $(u,b,x)$ is a straight. Since $\Gamma$ is green, $(a,u,b,x,d)$ is a root for
all $d\in\Gamma_x\backslash\{b\}$. Thus $\Gamma_x\cap \Gamma_2(a)=\{b\}$. We conclude
that either $a\in \Gamma_x$ or
$\Gamma_x\subset \Gamma_2(a)$ or $|\Gamma_x\cap \Gamma_2(a)|=1$. Hence $S_\Gamma$ is a
polar space. Since $\Gamma$ is $2$-plump, $S_\Gamma$ is thick. By \ref{toy20},
$S_\Gamma$ is non-degenerate. Thus (i) holds. The assertion~(ii) holds by definition
and (iii) follows from (ii) and \ref{hau12}.
\end{proof}

In the next result, we produce examples of green Veldkamp quadrangles that
are not flat.

\begin{proposition}\label{hau14}
Let $S=(P,L)$ be a thick non-degenerate polar space, suppose that ${\rm rk}(S)\ge3$ and 
let $\Gamma=\Gamma_S$ be as in {\rm\ref{hau12}}. Let $d\in P$,
let $P_1=d^\perp\backslash\{d\}$ and let $L_1=\{x\in L\mid x\subset P_1\}$.
Let $\Omega$ be the subgraph of $\Gamma$ spanned by $P_1\cup L_1$ endowed
with the natural restriction of the local opposition relations of $\Gamma$ to
the neighborhoods of $\Omega$. Then the following hold:
\begin{enumerate}[\rm(i)]
\item $\Omega$ is a green Veldkamp quadrangle that is not flat.
\item Let $a,b\in P_1$. Then $a^{\rm op}=b^{\rm op}$ in $\Omega$ if and only if
$a$, $b$ and $d$ are collinear.
\item Let $e\in d^{\rm op}$, let $P_2=e^\perp\cap P_1$ and let $L_2$ be the set
of lines contained in $P_2$. Then $L_2\ne\emptyset$, $S_2:=(P_2,L_2)$ is a 
thick non-degenerate polar space and $\Gamma_{S_2}$ is, up to isomorphism, 
the flat quotient of $\Omega$.
\end{enumerate}
\end{proposition}

\begin{proof}
We first want to show that $\Omega$ is $2$-plump. Let $a\in P_1$, let $u=ad$ and choose 
$x,y\in\Omega_a$. Since $\Gamma_x\subset P_1$, it follows from \eqref{hau12a}
that the path $(u,a,x)$ is not straight.
Hence $E:=\langle u\cup x\rangle$ is a singular plane containing $u$. 
Similarly, $F:=\langle u\cup y\rangle$ is a singular plane containing $u$.
By \ref{hau22z}, there exists a singular plane $G$ containing $u$ such
that neither $\langle E\cup G\rangle$ nor $\langle F\cup G\rangle$ is singular.
Let $z$ be a line in $G$ that intersects $u$ in the point $a$. 
Since $G$ is singular, the path $(u,a,z)$ is not straight. 
By \eqref{hau12c}, we have $\Gamma_z\subset P_1$. In other words, $z\in L_1$.
If $(x,a,z)$ were not straight, then $\langle x\cup z\rangle$ would be singular, but 
by \cite[7.4.7(v)]{BC}, this would imply that $\langle E\cup G\rangle$ is singular.
Hence $(x,a,z)$ is straight. Similarly,
$(y,a,z)$ is straight. Thus local opposition at $a$ is $2$-plump. Since 
$\Gamma$ is thick, so is $\Omega$. Therefore $\Omega$ is $2$-plump.

Again let $a\in P_1$ and $u=ad$. This time let $x$ be an arbitrary element of $L_1\backslash\Omega_a$.
By \ref{toy21}, there are $3$-paths in $\Gamma$ from $a$ to $x$.
Let $(a,y,b,x)$ be one of them. Suppose that $y\ne u$. Since $x\in L_1$, we have $b\in d^\perp$.
By \ref{toy20}, therefore, $d$ and $b$ are not opposite, so $(d,u,a,y,b)$ is not a root.
Hence $(d,u,a,y)$ is not straight and thus $y\in L_1$ by \eqref{hau12c}. Thus the distance from
$a$ to $x$ in $\Omega$ is~$3$.

Suppose now that $y=u$. Since $x\in L_1$, 
the path $(d,u,b,x)$ is not straight by \eqref{hau12a}. Hence $(a,u,b,x)$ is not straight.
Therefore $\Gamma_x\subset\Gamma_2(a)$ by \eqref{hau12c}. Let $c\in\Gamma_x\backslash\{b\}$.
Then there exists a $3$-path $(a,v,c,x)$ in $\Omega$
from $a$ to $x$ containing $c$ and $v\ne u$.
By \ref{fla40}, $(u,a,v)$ is not straight. By \eqref{hau12c},
$v\in L_1$. Therefore the distance from $a$ to $x$ in $\Omega$ is~$3$
also in this case.
By \ref{fla3}, however, $(a,u,b)$ is the only $2$-path in $\Gamma$ from $a$ to $b$,
so the distance from $a$ to $b$ in $\Omega$ is not~$2$. Since $b\in\Omega_x$, we conclude
that the distance from $a$ to $b$ in $\Omega$ is~$4$ and that $(a,v,c,x,b)$ is a non-straight
$4$-path from $a$ to $b$ in $\Omega$.

We have thus shown that 
\begin{equation}\label{hau40}
{\rm diam}(\Omega)=4.
\end{equation}
In particular, $\Omega$ is connected. Therefore $\Omega$ satisfies
(VP1). Since $\Gamma$ satisfies (VP2), so does $\Omega$. 

Now suppose that $(a,x,b,y,c,z)$ is a straight $5$-path in $\Omega$. 
By (VP2) and (VP3), there is a unique straight $3$-path $\alpha$ from $a$ to $z$ in $\Gamma$.
By \eqref{hau40}, the distance from $a$ to $z$ in $\Omega$ is also~$3$. 
By \ref{toy20}, ${\rm dist}(a,c)=4$. Thus $\Gamma_z\not\subset\Gamma_2(a)$.
By \eqref{hau12a}, it follows that $\alpha$ lies in $\Omega$.
Hence $\Omega$ satisfies (VP3). Thus $\Omega$ is a green Veldkamp quadrangle. 

Next let $a,b\in P_1$, let $u=da$ and suppose that $b\in\Gamma_u\backslash\{a\}$.
By \ref{fla3}, there is no $x\in L_1$ such that $\Gamma_x$ contains $a$ and $b$.
By \eqref{hau40}, it follows that the distance from $a$ to $b$ in $\Omega$ is~$4$.
By \ref{fla40}, there is no root from $a$ to $b$. Hence every $4$-path from
$a$ to $b$ in $\Omega$ is a weed in $\Omega$. 
By \ref{fla5x}(i), we conclude that $a^{\rm op}=b^{\rm op}$ in $\Omega$.

Now let $a,b\in P_1$, let $u=ad$ and suppose that $a\ne b$ and $a^{\rm op}=b^{\rm op}$ in $\Omega$.
By \ref{fla10}(ii) and \ref{fla5x}(i), the distance from $a$ to 
$b$ in $\Omega$ is~$4$ and every $4$-path from $a$ to $b$ in $\Omega$ is a weed in $\Omega$. 
Let $(a,y,c,x,b)$ be a $4$-path from $a$ to $b$ in $\Omega$. Thus $(a,y,c,x)$
is not straight and by \ref{fla42}, $b$ is the unique element of $\Gamma_x$ that
is not at distance~$2$ from $a$ in $\Omega$. 
By \eqref{hau12a}, $b\in\Gamma_x\subset\Gamma_2(a)$. Hence 
$b\in\Gamma_u$. We conclude that~(i) and~(ii) hold.

Finally, we let $e$ and $S_2=(P_2,L_2)$ be as in~(iii). 
By \ref{hau22}, $S_2$ is a thick
non-degenerate polar space. Let $\bar\Omega$ denote
the flat quotient of $\Omega$ and let $S_{\bar\Omega}$ 
be as in \ref{hau13}(i). Thus by \ref{fla12}, the point set of $S_{\bar\Omega}$ is the
set of equivalence classes $[a]$ for all $a\in P_1$ as defined in \ref{fla70}
(with $\Omega$ in place of $\Gamma$).
By (BS), for each $w\in\Gamma_d$,
there exists a unique $a\in\Gamma_w\cap P_2$. By (ii), therefore, there exists
a cannonical bijection $\psi$ from the point set of $S_{\bar\Omega}$ to $P_2$ mapping
the equivalence class $[a]$ 
to the unique vertex in $\Gamma_{ad}\cap P_2$ for all $a\in P_1$.
Let $a,b\in P_1$ and suppose that the points $[a],[b]$ of $S_{\bar\Omega}$
are distinct. Let $u=ad$ and $v=bd$ and let $a'=\psi([a])$ and $b'=\psi([b])$. 
Suppose first that $[a]$ and $[b]$ are collinear in $S_{\bar\Omega}$.
By \ref{fla20}(ii), we can assume that the distance between $a$ and $b$ in $\Omega$ is~$2$.
Let $x=a\wedge b$.  Since $x$ is a vertex of $\Omega$, we have $x\in\Gamma_3(d)$.
Hence $(a,u,d,v,b,x,a)$ is a closed $6$-circuit. By \ref{fla40}, therefore, $(u,d,v)$ is
not straight. Therefore also $(a',u,d,v,b')$ is not
straight. By \ref{hau12}, $\Gamma$ is flat. By \ref{fla32}, therefore,
$(a',u,d,v,b')$ is not a weed. Hence 
${\rm dist}(a',b')=2$. It follows that
$a'$ and $b'$ are collinear in $S_2$. Thus $\psi$ maps collinear points of $S_{\bar\Omega}$
to collinear points of $S_2$. Conversely, if $a'$ and $b'$ are collinear in $S_2$, then 
(since $L_2\subset L_1$), $[a']=[a]$ and $[b']=[b]$ are collinear in $S_{\bar\Omega}$.
We conclude that $[a]$ and $[b]$ are collinear in $S_{\bar\Omega}$ if and only if
$a'$ and $b'$ are collinear in $S_2$. By \cite[7.4.12]{BC}, it follows that
$\psi$ is an isomorphism from $S_{\bar\Omega}$ to $S_2$. By \ref{hau13},
it follows that $\psi$ induces an isomorphism from $\bar\Omega$ to $\Gamma_{S_2}$. Thus (iii) holds.
\end{proof}

\section{Tits quadrangles}

In this section, we give the classification of green Veldkamp quadrangles
that are not generalized quadrangles in \ref{chin5} and \ref{chin6}.
In particular, we show that these Veldkamp quadrangles are all, in fact, Tits quadrangles
as defined in \cite[1.1.6 and 1.1.8]{TP}.

\begin{definition}\label{chin1}
Let $\Gamma$ be a Veldkamp $n$-gon for some $n\ge3$ and let
$$\alpha=(x_0,x_1,\ldots,x_n)$$ 
be a root. We denote by $U_\alpha$ the pointwise stabilizer in ${\rm Aut}(\Gamma)$ of the set 
$$\Gamma_{x_1}\cup\Gamma_{x_2}\cup\cdots\cup\Gamma_{x_{n-1}}.$$
The subgroup $U_\alpha$
is called the {\it root group} of $\Gamma$ associated with $\alpha$. 
\end{definition}

\begin{definition}\label{chin1x}
Let $\Gamma$ be a Veldkamp $n$-gon for some $n\ge3$. Then $\Gamma$ is {\it Moufang}
if for each root $\alpha=(x_0,x_1,\ldots,x_n)$, the root group $U_\alpha$ acts 
transitively on the set of vertices in $\Gamma_{x_n}$ that are opposite $x_{n-1}$
at $x_n$.
\end{definition}

\begin{proposition}\label{chin1y}
Let $\Gamma$ be a Veldkamp graph, let $(x,y,z)$ be a straight $2$-path, let
$M$ be a subgroup of the stabilizer of $(x,y)$ in ${\rm Aut}(\Gamma)$ and let
$g$ be an element of ${\rm Aut}(\Gamma)$ mapping $(y,x)$ to $(y,z)$.
Suppose that $M$ acts transitively on the set of vertices in $\Gamma_y$ that are
opposite $x$. Then $\langle M,M^g\rangle$ acts transitively on $\Gamma_y$.
\end{proposition}

\begin{proof}
The conjugate $M^g$ fixes $z$ and acts transitively on the set of
vertices of $\Gamma_y$ opposite $z$. Since $\Gamma$ is $2$-plump,
it follows that $\langle M,M^g\rangle$ acts transitively on $\Gamma_y$.
\end{proof}

\begin{remark}\label{chin2}
Let $\Gamma$ be a Veldkamp $n$-gon for some $n\ge3$. Then $\Gamma$ is Moufang
if and only if it is a Tits $n$-gon as defined in \cite[1.1.6 and 1.1.8]{TP}.
\end{remark}

\begin{example}\label{chin3}
Let $\Lambda$ be either 
a quadratic space 
$$(K,L,q)$$ 
(Case~I) or 
a pseudo-quadratic space 
$$\Lambda=(K,K_0,\sigma,L,q)$$ 
as defined in \cite[11.15--11.17]{TW} (Case~II). 
In Case~II, $L$ is a right vector
space over $K$ and the associated sesquilinear form $f$ on $L$ is assumed to be
skew-hermitian. 
In Case~I,we write scalars on the left, we
let $f$ denote the bilinear form associated with $q$ and we let 
$\sigma$ denote the trivial automorphism of $K$. 
In both cases, we assume that $\Lambda$ is non-degenerate as defined in 
\cite[8.2.3]{spherical}. Thus
$$\{a\in L\mid q(a)=f(a,a)=0\}=\{0\}$$
in Case~I and 
$$\{a\in L\mid f(a,a)=0\text{ and }q(a)\in K_0\}=\{0\}$$
in Case~II. In Case~II, $L=0$ is allowed, but in Case~I, $L\ne0$. 
In Case~II we assume that ${\rm char}(K)\ne2$ if $K_0=K$. 
We do not make any restriction on the Witt index of $q$ in Case~I or in Case~II.

It is important for the reader to keep in mind 
that $\{t+t^\sigma\mid t\in K\}\subset K_0$ in Case~II.
Many of the claims we make below depend on this observation.

Let $V=K^4\oplus L$ and let $Q\colon V\mapsto K$ be given by
$$Q(t_1,t_1',t_2,t_2',a)=t_1^\sigma t_1'+t_2^\sigma t_2'+q(a)$$
for all $(t_1,t_1',t_2,t_2',a)\in V$.
In Case~I, let
$$F\big((s_1,s_1',s_2,s_2',b),(t_1,t_1',t_2,t_2',a)\big)=s_1t_1'+t_1s_1'+s_2t_2'+t_2s_2'+f(a,b)$$
and in Case~II, let
$$F\big((s_1,s_1',s_2,s_2',b),(t_1,t_1',t_2,t_2',a)\big)=s_1^\sigma t_1'-(s_1')^\sigma t_1
+s_2^\sigma t_2'-(s_2')^\sigma t_2+f(b,a)$$
for all $(s_1,s_1',s_2,s_2',b),(t_1,t_1',t_2,t_2',a)\in V$.
In Case~I, $Q$ is a non-degenerate quadratic form on $V$
and $F$ is the bilinear form associated with $Q$. In Case~II,
$Q$ is a non-degenerate pseudo-quadratic form and $F$ is the skew-hermitian form 
associated with $Q$. In particular, the image of $Q$ is defined only modulo $K_0$ and when we 
say that $Q$ vanishes on a given subspace $V_1$ of $V$, we mean that $Q(V_1)\subset K_0$.
If we are in Case~I or in Case~II with $K_0\ne K$, then
by \cite[11.19]{TW}, the form $F$ is uniquely determined by
$Q$ and, in particular, $F$ is identically zero on every subspace on which $Q$ vanishes
identically. If we are in Case~II with $K_0=K$,
then $Q$ vanishes identically on $V$,
$K$ is commutative, $\sigma=1$, $F$ is alternating and ${\rm char}(K)\ne2$ (by hypothesis).
We call this subcase of Case~II the {\it symplectic} case. 

Let $T$ denote the additive group of $L$ 
in Case~I and in Case~II, let $T$ denote
the group with underlying set 
$$\{(a,t)\in L\times K\mid q(a)-t\in K_0\}$$
defined in \cite[11.24]{TW}. Thus $T=K_0$ if $L=0$ and 
$$(a,t)(b,u)=\big(a+b,t+u+f(b,a)\big)$$
and 
$$(a,t)^{-1}=(-a,-t^\sigma)$$
for all $(a,t),(b,u)\in T$ otherwise.
Note that in the symplectic case, the condition $q(a)-t\in K_0$ is vacuous and the underlying set of
$T$ is simply the direct product $L\times K$.

For each $i$, let $W_i$ be the set of subspaces
of $V$ of dimension~$i$ on which $Q$ and $F$ vanish identically.
For $i\ge2$, we identify each element of $W_i$ with the set of elements of $W_1$ it contains.
This makes $(W_1,W_2)$ into a thick line space.
If $V_1\in W_1$ and $V_2\in W_2$, then $\dim_K(V_2\cap V_1^\perp)=1$ or $2$.
(Here orthogonal means orthogonal with respect to the form $F$.) 
Thus the line space $(W_1,W_2)$ is, in fact, a polar space as defined in \ref{hau7}. 
Since $\Lambda$ is non-degenerate, we have $V_1^\perp\ne V$.
Hence the polar space 
$(W_1,W_2)$ is non-degenerate as defined in \ref{hau9}. We denote it by $S_\Lambda$. For each $i$,
the $i$-dimensional singular subspaces of $S=S_\Lambda$ are the subspaces in $W_{i+1}$.
Let $\Gamma=\Gamma_S$ be the corresponding Veldkamp quadrangle as defined in \ref{hau13}.

We want to show that $\Gamma$ is a Tits quadrangle and that its
root group sequences are as described in \cite[16.3]{TW} in Case~I and in \cite[16.2 and 16.5]{TW}
in Case~II, but without the assumption in either case that $q$ is anisotropic. 

For each isotropic vector $u$ in $V$ (i.e.~for each vector $u$ that spans an element of $W_1$)
we will denote simply by $u$ the element of $W_1$ spanned by $u$ when it is clear that we are
referring vertices of $\Gamma$. For each pair of orthogonal 
isotropic vectors $u$ and $v$, we will denote
by $uv$ the element of $W_2$ they span. Next we set
$$x_1=(1,0,0,0,0),\ x_1'=(0,1,0,0,0)$$ 
as well as
$$x_2=(0,0,1,0,0)\text{ and }x_2'=(0,0,0,1,0)$$
in $V$ and we identify $L$ with its image in $V$ under the map $u\mapsto(0,0,0,0,u)$.
Let $\Sigma$ be the straight $8$-circuit of $\Gamma$ with vertex set
$$x_1,x_1x_2',x_2',x_1'x_2',x_1',x_1'x_2,x_2,x_1x_2$$
and let 
$$\alpha=(x_2',x_1x_2',x_1,x_1x_2,x_2)\text{ and }\gamma=(x_1'x_2',x_2',x_1x_2',x_1,x_1x_2)$$
as well as 
$$\beta=(x_1',x_1'x_2',x_2',x_1x_2',x_1)\text{ and }\delta=(x_1'x_2,x_1',x_1'x_2',x_2',x_1x_2').$$
Note that the roots $\alpha,\ldots,\delta$ all contain the edge $\{x_1x_2',x_2'\}$. 

\begin{notation}\label{foo0}
Let $\varepsilon=1$ in Case~I and $\varepsilon=-1$ in Case~II.
Let $\dot\rho$ denote the linear automorphism of $V$ that interchanges $x_1$ with $\varepsilon x_2'$ 
and $x_1'$ with 
$x_2$ and acts trivially on $L$. Let $\dot\tau$ denote the linear automorphism of
$V$ that interchanges
$x_1$ with $\varepsilon x_1'$, fixes $x_2$ and $x_2'$ and acts trivially on $L$,
Both $\dot\rho$ and $\dot\tau$
are isometries of both $Q$ and $F$ and hence both induce automorphisms of $\Gamma$.
We denote thee elements by $\rho$ and $\tau$. 
Both $\rho$ and $\tau$ stabilize $\Sigma$, $\rho$ fixes the vertices $x_1x_2'$ and $x_1'x_2$,
$\tau$ fixes the vertices $x_2$ and $x_2'$ and $\langle\rho,\tau\rangle$ acts transitively
on the edge set of $\Sigma$. Hence every root contained in $\Sigma$
is in the same $\langle\rho,\tau\rangle$-orbit as either $\alpha$ or $\delta$.
\end{notation}

\begin{notation}\label{foo1}
Suppose we are in Case~I.
For all $a,b\in L$, let $\dot\alpha_a$ be the linear automorphism of $V$ 
that fixes $x_1$, $x_2$ and $x_2'$, maps $x_1'$ to $x_1'-q(a)x_1+a$
and maps an arbitrary element $u\in L$ to $u-f(u,a)x_1$ and 
let $\dot\beta_b$ be the linear automorphism of $V$ that fixes $x_1$, $x_1'$ and $x_2'$,
maps $x_2$ to $x_2-q(b)x_2'+b$ and maps an arbitrary element $u\in L$ to 
$u-f(u,b)x_2'$. For all $s,t\in K$, let $\dot\gamma_t$ be the linear automorphism of $V$
that fixes $x_1$ and $x_2'$, acts trivially on $L$, maps
$x_1'$ to $x_1'-tx_2'$ and $x_2$ to $x_2+tx_1$ and let $\dot\delta_s$
be the linear automorphism of $V$ that fixes $x_1'$ and $x_2'$, acts
trivially on $L$ and maps $x_1$ to $x_1+sx_2'$ and $x_2$ to $x_2-sx_1'$.
\end{notation}

\begin{notation}\label{foo2}
Suppose that we are in Case~II.
For all $(a,t),(b,s)\in T$, let $\dot\alpha_{(a,t)}$ be the linear automorphism of $V$ 
that fixes $x_1$, $x_2$ and $x_2'$, maps $x_1'$ to $x_1'+x_1t+a$
and maps an arbitrary element $u\in L$ to $u+x_1f(a,u)$ and 
let $\dot\beta_{(b,s)}$ be the linear automorphism of $V$ that fixes $x_1$, $x_1'$ and $x_2'$,
maps $x_2$ to $x_2-x_2's-b$ and maps an arbitrary element $u\in L$ to 
$u+x_2'f(b,u)$. For all $s,t\in K$, let $\dot\gamma_t$ be the linear automorphism of $V$
that fixes $x_1$ and $x_2'$, acts trivially on $L$, maps
$x_1'$ to $x_1'+x_2't^\sigma$ and $x_2$ to $x_2-x_1t$ and let $\dot\delta_s$
be the linear automorphism of $V$ that fixes $x_1'$ and $x_2'$, acts
trivially on $L$ and maps $x_1$ to $x_1-x_2's^\sigma$ and $x_2$ to $x_2-x_1's$.
\end{notation}

For all $v,w\in T$
and all $s,t\in K$, the maps $\dot\alpha_v$, $\dot\beta_w$, $\dot\gamma_t$
and $\dot\delta_s$ defined in \ref{foo1} and \ref{foo2} are all isometries of $Q$ and $F$
and hence induce automorphisms of $\Gamma$. We
denote these automorphisms by $\alpha_v$, $\beta_w$, $\gamma_t$ and $\delta_s$.
Let 
$$A=\{\alpha_v\mid v\in T\},\ B=\{\beta_w\mid w\in T\},\ C=\{\gamma_t\mid t\in K\}\text{ and }
D=\{\delta_s\mid s\in K\}.$$

\begin{observations}\label{foo3}
Let $K^+$ denote the additive group of $K$.
With a bit of calculation, it can be checked that the following hold:
\begin{enumerate}[\rm(i)]
\item $A$, $B$, $C$ and $D$ are subgroups of ${\rm Aut}(\Gamma)$.
\item $A\subset U_\alpha$, $B\subset U_\beta$, $C\subset U_\gamma$ and $D\subset U_\delta$.
\item $v\mapsto\alpha_v$ and $v\mapsto\beta_v$ are isomorphisms from $T$ to $A$ and from $T$ to $B$
and $t\mapsto\gamma_t$ and $t\mapsto\delta_t$ are isomorphisms from $K^+$
to $C$ and from $K^+$ to $D$.
\item $A$ acts transitively on the set of vertices adjacent
to $x_2$ that are opposite $x_1x_2$ at $x_2$ and on the set of vertices adjacent to $x_2'$
that are opposite $x_1x_2'$ at $x_2'$.
\item $B$ acts transitively on the set of
vertices adjacent to $x_1$ that are opposite $x_1x_2'$ at $x_1$ and on the set of
of vertices adjacent to $x_1'$ that are opposite $x_1'x_2'$ at $x_1'$.
\item $C$ acts transitively on 
$\Gamma_{x_1x_2}\backslash\{x_1\}$ and on $\Gamma_{x_1'x_2'}\backslash\{x_2'\}$.
\item $D$ acts transitively on $\Gamma_{x_1x_2'}\backslash\{x_2'\}$ and on
$\Gamma_{x_1'x_2}\backslash\{x_1'\}$.
\end{enumerate}
\end{observations}
Now let $M$ denote the group generated by $\rho$, $\tau$, $A$, $B$, $C$ and $D$.
By \ref{foo3}(iv)--(vii), the subgroup $\langle A,B,C,D\rangle$ of $M$ acts transitively 
both on the set
of roots of $\Gamma$ starting with $(x_2',x_1x_2')$ and on the set of roots of $\Gamma$
starting with $(x_1x_2',x_2')$. 
By \ref{chin1y} and \ref{foo3}(iv) and (vii), the subgroup $\langle A,\tau\rangle$ acts transitively
on $\Gamma_{x_2'}$ and the subgroup $\langle C,\rho\rangle$ acts transitively
in $\Gamma_{x_1x_2'}$. Since $x_2'$ and $x_1x_2'$ are adjacent vertices and $\Gamma$ is connected,
it follows that $M$ acts transitively on the set of edges of $\Gamma$. 
By \ref{foo0}, we conclude that 
every root of $\Gamma$ is in the same $M$-orbit as either $\alpha$ or $\gamma$.
Hence by \ref{chin1x} and \ref{foo3}(ii), (iv) and (vii),
$\Gamma$ is Moufang. By \ref{chin2}, therefore, $\Gamma$ is a Tits quadrangle.
By \ref{foo3}(ii) and \cite[1.3.25]{TP}, finally, we have
\begin{equation}\label{chin3a}
A=U_\alpha,\ B=U_\beta,\ C=U_\gamma\text{ and }D=U_\delta.
\end{equation}
\end{example}

\begin{definition}\label{chin4}
Let $\Lambda$, $T$ and $f$ be as in 
{\rm\ref{chin3}} (so $\Lambda$ is non-degenerate
and if $K_0=K$ in Case~II, then ${\rm char}(K)\ne2$) and let $K^+$ be as in \ref{foo3}.
Let $\Delta$ be a Tits quadrangle, let $(\Sigma,i\mapsto w_i)$ be a coordinate
system for $\Delta$ and let $i\mapsto U_i$ be the corresponding root group 
labeling as defined in \cite[2.3]{tq}. The root groups $U_1,\ldots,U_4$ satisfy
the commutator conditions in \cite[2.5]{tq}. In particular, $[U_i,U_{i+1}]=1$ for all $i$. 
\begin{enumerate}[\rm(i)]
\item $\Delta$ is called {\it orthogonal of type $\Lambda$} if $\Lambda=(K,L,q)$ is
a quadratic space and, after replacing the coordinate system by its opposite 
(as defined in \cite[2.3]{tq}) if necessary, there exist
isomorphisms $x_i\colon K^+\to U_i$ for $i=1$ and $3$ and $x_i\colon T\to U_i$
for $i=2$ and $4$ such that
\begin{align*}
[x_2(a),x_4(b)^{-1}]&=x_3(f(a,b))\text{ and}\\
[x_1(t),x_4(a)^{-1}]&=x_2(ta)x_3(tq(a))
\end{align*}
for all $t\in K$ and all $a,b\in L$ and $[U_1,U_3]=1$.
\item $\Delta$ is called {\it pseudo-quadratic of type $\Lambda$} if
$\Lambda=(K,K_0,\sigma,L,q)$ is a pseudo-quadratic space 
and, after replacing the coordinate system by its opposite if necessary,
there exist
isomorphisms $x_i\colon T\to U_i$ for $i=1$ and $3$ and $x_i\colon K^+\to U_i$
for $i=2$ and $4$ such that 
\begin{align*}
[x_1(a,t),x_3(b,s)^{-1}]&=x_2(f(a,b)),\\
[x_2(v),x_4(w)^{-1}]&=x_3(0,\ v^\sigma w+w^\sigma v)\text{ and}\\
[x_1(a,t),x_4(v)^{-1}]&=x_2(tv)x_3(av,\ v^\sigma tv)
\end{align*}
for all $(a,t),(b,s)\in T$ and all $v,w\in K$. 
\end{enumerate}
\end{definition}

Here now is our existence result for orthogonal and pseudo-quadratic Tits quadrangles.

\begin{theorem}\label{chin5}
Suppose that $\Lambda$ either a non-degenerate quadratic space 
or that $\Lambda=(K,K_0,\sigma,L,q)$ is a non-degenerate pseudo-quadratic space
such that ${\rm char}(K)\ne2$ if $K_0=K$ and 
let $\Gamma$ be as in {\rm\ref{chin3}}. 
Then $\Gamma$ an orthogonal Tits quadrangle of type $\Lambda$ if $\Lambda$ is a quadratic space
and $\Gamma$ is an pseudo-quadratic Tits quadrangle of type $\Lambda$ if $\Lambda$ is 
a pseudo-quadratic space.
\end{theorem}

\begin{proof}
We continue with all the notation in \ref{chin3}, where we already showed that
$\Gamma$ is a Tits quadrangle. In Case~I, let
$$(w_1,w_2,\ldots,w_8)=(x_1'x_2,x_1',x_1'x_2',x_2',x_1x_2',x_1,x_1x_2,x_2),$$
let $x_1(t)=\delta_t$ and $x_3(t)=\gamma_t$
for all $t\in K$ and let $x_2(a)=\beta_a$ and $x_4(a)=\alpha_a$ for all $a\in L$.
In Case~II, let 
$$(w_1,w_2,\ldots,w_8)=(x_2,x_1x_2,x_1,x_1x_2',x_2',x_1'x_2',x_1',x_1'x_2),$$
let  $x_1(a,t)=\alpha_{(a,t)}$ 
and $x_3(a,t)=\beta_{(a,t)}$ for all $(a,t)\in T$ and let 
$x_2(v)=\gamma_v$ and $x_4(v)=\delta_v$ for all $v\in K$. 
In both Case~I and Case~II, $(\Sigma,i\mapsto w_i)$ is a coordinate system of $\Gamma$.
Let $i\mapsto U_i$ be the corresponding root group labeling.
By \ref{chin3}(iii) and \eqref{chin3a},
$x_i$ is injective and its image is $U_i$ for all $i\in[1,4]$ both in Case~I and in Case~II.
It now requires only the formulas in \ref{foo1}, \ref{foo2} and a bit of patience to check 
that the commutator relations in (i) and (ii) hold.
\end{proof}

We now describe two further families of thick non-degenerate polar spaces of rank at least~$3$.

\begin{definition}\label{chin50}
Let $S$ be the thick non-degenerate polar space of rank~$3$ described in \cite[7.13]{spherical}
and let $\Gamma=\Gamma_S$ denote the corresponding Veldkamp quadrangle.
Then $\Gamma$ is a Tits quadrangle and there exist a skew-field $K$,
a coordinate system $(\gamma,i\mapsto w_i)$ with root group labeling $i\mapsto U_i$
and isomorphisms $x_i\colon K^+\to U_i$ for $i=1$ and $3$ 
and $x_i\colon K^+\oplus K^+\to U_i$ for $i=2$ and $4$ such that
\begin{align*}
[x_1(u),x_4(s,t)^{-1}]&=x_2(su,ut)x_3(sut)\\
[x_2(s,t),x_4(u,v)^{-1}]&=x_3(ut+sv)
\end{align*}
for all $u,v,s,t\in K$ as well as $[U_1,U_3]=[U_1,U_2]=[U_2,U_3]=[U_3,U_4]=1$. 
We call $\Gamma$ the Tits quadrangle {\it of type $D_3$} over $K$.
\end{definition}

\begin{remark}\label{chin50x}
Let $S$, $\Gamma$ and $K$ be as in \ref{chin50}.
If $K$ is commutative, then
$\Gamma$ is, in fact, of orthogonal type $\Lambda$ where $\Lambda$ is the hyperbolic
quadratic space of dimension~$2$ over $K$. When $K$ is non-commutative,
$S$ is non-embeddable and $\Gamma$ is not of orthogonal or pseudo-quadratic
type.
\end{remark}

\begin{definition}\label{chin60}
Let $\Gamma$ be the Tits quadrangle
obtained by applying the construction described 
in \cite[1.2.12 and 1.2.28]{TP} to the pair $(\Delta,T)$, where $\Delta$ is the 
exceptional building $C_3(J)$ for some octonion division algebra $J$ obtained by Galois
descent from the Tits index 
$$\includegraphics[scale=.7]{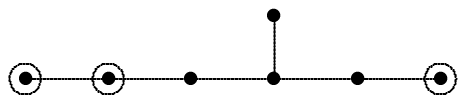}$$
and $T$ is the Tits index 
$$\includegraphics[scale=.7]{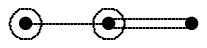}$$
and ${\rm rk}(S_\Gamma)=3$. We call $\Gamma$ the Tits quadrangle {\it of type $E_7$} over $J$.
Let $K$ denote the center of $J$.
Then the commutator relations associated with $\Gamma$ are
as described in \cite[8.2]{tq} with $\Xi$ set equal to the quadrangular algebra 
${\mathscr Q}_2(J,K)$ described in \cite[\S4]{alg}. 
\end{definition}

We can now formulate the classification of flat green Veldkamp quadrangles.

\begin{theorem}\label{chin6}
Let $\Gamma$ be a flat green Veldkamp quadrangle. Then $\Gamma$ is either 
a generalized quadrangle or one of the following holds:
\begin{enumerate}[\rm(i)]
\item $\Gamma$ is a Tits quadrangle of orthogonal of type $\Lambda$ as defined in {\rm\ref{chin4}}
for some non-degenerate quadratic space $\Lambda$
with Witt index at least~$3$.
\item $\Gamma$ is a Tits quadrangle of pseudo-quadratic of type $\Lambda$ as defined in {\rm\ref{chin4}}
for some non-degenerate quadratic space $\Lambda$
with Witt index at least~$3$.
\item $\Gamma$ is a Tits quadrangle of type $D_3$ or $E_7$ as defined in {\rm\ref{chin50}}
and {\rm\ref{chin60}}.
\end{enumerate}
In particular, if $\Gamma$ is not a generalized polygon, it is Moufang as
defined in {\rm\ref{chin1x}}.
\end{theorem}

\begin{proof}
If $\Gamma$ is not a generalized quadrangle, then by
\ref{hau12}, the rank of the polar space $S_\Gamma$ described in \ref{hau13} is at least~$3$.
By \cite[7.9.7]{shult}, $\Gamma$ is as in (i) or (ii) if ${\rm rk}(S_\Gamma)\ge4$
including the case that ${\rm rk}(S_\Gamma)$ is infinite. We can thus suppose
that ${\rm rk}(S_\Gamma)=3$. In this case, the claim holds by \cite[Thms.~8.22 and~9.1]{spherical}.
Alternatively, we can cite \cite[40.17 and 40.25(ii)-(iv)]{TW} for 
the classification of thick buildings of type~$C_3$. If $\Gamma$ is not of type $D_3$,
then by \cite[Prop.~7.13]{spherical}, the building of type~$C_3$ in \cite[Thm.~7.4]{spherical}
is thick. It follows from the classification of thick buildings of type $C_3$ 
that either $\Gamma$ is as in (i) or (ii) or $\Gamma$ is of type $E_7$.
\end{proof}

\begin{remark}\label{chin6a}
Moufang quadrangles (i.e.~generalized polyons satisfying the Moufang condition)
were classified in \cite[Thm.~17.4]{TW}. The Moufang quadrangles
described in \cite[16.2, 16.3 and 16.5]{TW} are precisely the Tits quadrangles of orthogonal and
pseudo-quadratic type as defined in \ref{chin4} for $\Lambda$ is anisotropic,
but there are also three further families:
the indifferent quadrangles described in \cite[16.4]{TW}, the exceptional quadrangles of type
$E_6$, $E_7$ and $E_8$ described in \cite[16.6]{TW} and the exceptional quadrangles of type
$F_4$ described in \cite[16.7]{TW}. There are, of course, also many generalized quadrangles
that do not satisfy the Moufang condition; see, for example,~\cite{hendrik}.
\end{remark}

\begin{remark}\label{fin1}
Let $\Gamma$ be a Tits $n$-gon for some $n$. As in \ref{chin4}, 
we choose a coordinate system $(\gamma,i\mapsto w_i)$ and
let $i\mapsto U_i$ be the corresponding root group labeling.
Let $U_i^\sharp$ for all $i$ and $\mu_\gamma$
be as in \cite[2.8 and 2.9]{tq} and let
$$H_i=\langle\mu_\gamma(a)\mu_\gamma(b)\mid a,b\in U_i^\sharp\rangle$$
for each $i$. As in \cite[2.23]{orth}, we say that $\Gamma$ is {\it razor-sharp} if
for each $i$, every nontrivial $H_i$-invariant normal subgroup of the root group $U_i$
contains elements of $U_i^\sharp$. 
In \cite[1.6.14]{TP}, we showed that if $\Gamma$ is razor-sharp, then
$n=3$, $4$, $6$ or $8$. In \cite{TP}, \cite{triangles} and
\cite{oct}, razor-sharp Tits triangles, hexagons and octagons were classified.
Partial results on the case $n=4$ can be found in \cite{tq} and \cite{orth}.
In particular, it is shown in \cite[3.11]{orth} that every normal $5$-plump
razor-sharp Tits quadrangle
is orthogonal, where ``normal'' is defined in analogy to
the eponymous notion in \cite[21.7]{TW}. 
It is perhaps of interest to observe that many of the difficulties
in the proof of this result stem from not knowing that the underlying 
Veldkamp quadrangle is green and that this fact emerges in the proof of \cite[3.11]{orth} only
at the very end. (It is shown in \cite[3.10]{orth} that,
conversely, every orthogonal Tits quadrangle is razor-sharp except for those
that are of type~$D_3$; see \ref{chin50x}.)
\end{remark}

\begin{remark}\label{fin2}
Let $\Gamma$, $(\gamma,i\mapsto w_i)$ and $i\mapsto U_i$ be as in
\ref{chin50}. Then the following observations can be shown to hold:
There is an isomorphism $x_0\colon K^+\oplus K^+\to U_0$ such that
\begin{align*}
[x_0(s,t),x_3(u)^{-1}]&=x_1(tus)x_2(us,tu)\\
[x_0(s,t),x_2(u,v)^{-1}]&=x_1(tu+vs)
\end{align*}
for all $s,t,u,v\in K$. We have $x_4(s,t)\in U_4^\sharp$ 
if and only if $st\in K^\times$ and if $st\in K^\times$, then 
$$\mu(x_4(s,t))=x_0(t^{-1},s^{-1})x_4(s,t)x_0(t^{-1},s^{-1})$$
and conjugation by $\mu(x_4(s,t))$ induces the maps
\begin{align*}
x_0(u,v)&\mapsto x_4(svs,tut)\\
x_1(u)&\mapsto x_3(sut)\\
x_2(u,v)&\mapsto x_2(-svt^{-1},-s^{-1}ut)\\
x_3(u)&\mapsto x_1(s^{-1}ut^{-1})\\
x_4(u,v)&\mapsto x_0(t^{-1}vt^{-1},s^{-1}us^{-1})
\end{align*}
for all $u,v\in K$. Thus $\{x_4(u,0)\mid u\in K\}$ is a nontrivial $H_4$-invariant subgroup of $U_4$
that is disjoint from $U_4^\sharp$. Hence the Tits quadrangle $\Gamma$ is not razor-sharp.

For the sake of completeness, we record also the following:
There is an isomorphism $x_5\colon K^+\to U_5$ such that
$$[x_2(s,t),x_5(u)^{-1}]=x_3(sut)x_4(su,ut)$$
for all $s,t,u\in K$. We have $x_1(u)\in U_1^\sharp$ if and only if $u\in K^\times$ and 
if $u\in K^\times$, then
$$\mu_\gamma(x_1(u))=x_5(u^{-1})x_1(u)x_5(u^{-1})$$ 
and conjugation by $\mu_\gamma(x_1(u))$ induces the maps
\begin{align*}
x_1(s)&\mapsto x_5(u^{-1}su^{-1})\\
x_2(s,t)&\mapsto x_4(-su^{-1},-u^{-1}t)\\
x_3(s)&\mapsto x_3(s)\\
x_4(s,t)&\mapsto x_2(su,ut)\\
x_5(s)&\mapsto x_1(usu)
\end{align*}
for all $s,t\in K$.
\end{remark}

\end{document}